\documentclass[a4paper]{amsart}[12pt]

\usepackage{color}
\usepackage{amsxtra}
\usepackage{mathtools}
\usepackage{enumerate}
\usepackage{nicefrac}
\usepackage{psfrag}
\usepackage{ mathrsfs }
\usepackage{tikz}
\usetikzlibrary{shadings}
\usepackage{color}
\usepackage{hyperref}
\hypersetup{
  colorlinks   = true, 
  urlcolor     = blue, 
  linkcolor    = blue, 
  citecolor   = red 
}

\newtheorem{theorem}{Theorem}[section]
\newtheorem{lemma}{Lemma}[section]
\newtheorem{proposition}{Proposition}[section]

\newtheorem{corollary}{Corollary}
\newtheorem{remark}{Remark}[section]

\renewcommand{\theequation}{\arabic{section}.\arabic{equation}}
\newcommand{\R}{\mathbb{R}}
\newcommand{\ds}{\displaystyle}

\title[Symmetry results in the half space] {
Symmetry results in the half space for a semi-linear fractional Laplace equation\\ through a one-dimensional analysis
}

\author[B. Barrios]{B. Barrios}
	\address{B. Barrios \hfill\break\indent
		Departamento de An\'{a}lisis Matem\'{a}tico,
		Universidad de La Laguna\hfill 
		\break \indent C/. Astrof\'{\i}sico Francisco S\'{a}nchez s/n, 
		38200 -- La Laguna, SPAIN}
	\email{bbarrios@ull.es}

\author[L. Del Pezzo]{L. Del Pezzo}
	\address{L. Del Pezzo \hfill\break\indent
		CONICET  \hfill\break\indent
		Departamento de Matem\'atica y Estad\'istica,
		Universidad Torcuato Di Tella
		\hfill\break\indent Av. Figueroa Alcorta 7350 (C1428BCW)
		\hfill\break\indent Buenos Aires, ARGENTINA. }
	\email{ldelpezzo@utdt.edu}
	\urladdr{http://cms.dm.uba.ar/Members/ldpezzo/}

\author[J. Garc\'{\i}a-Meli\'{a}n]{J. Garc\'{\i}a-Meli\'{a}n}
 	\address{J. Garc\'{\i}a-Meli\'{a}n \hfill\break\indent
		Departamento de An\'{a}lisis Matem\'{a}tico, 
		Universidad de La Laguna
		\hfill \break \indent C/. Astrof\'{\i}sico 
		Francisco S\'{a}nchez s/n, 38200 -- La Laguna, SPAIN
		\hfill\break\indent
		{\rm and} \hfill\break
		\indent 
		Instituto Universitario de Estudios Avanzados (IUdEA) 
		\hfill\break\indent en F\'{\i}sica At\'omica,
		Molecular y Fot\'onica, 
		\hfill\break\indent Universidad de La Laguna
		\hfill\break\indent C/. Astrof\'{\i}sico Francisco
		S\'{a}nchez s/n, 38200 -- La Laguna, SPAIN.}
	\email{jjgarmel@ull.es}

\author[A. Quaas]{A. Quaas} 
	\address{A. Quaas\hfill\break\indent
		Departamento de Matem\'{a}tica, \hfill\break\indent
		Universidad T\'ecnica Federico Santa Mar\'{\i}a
		\hfill\break\indent  Casilla V-110, Avda. Espa\~na, 
		1680 -- Valpara\'{\i}so, CHILE.}
	\email{{\tt alexander.quaas@usm.cl}}

\begin{document}

\maketitle

\begin{abstract}  In this paper we analyze the semi-linear fractional Laplace equation
$$(-\Delta)^s u = f(u) \quad\text{ in } \mathbb{R}^N_+,\quad u=0 \quad\text{ in } \mathbb{R}^N\setminus \mathbb{R}^N_+,$$
where $\R^N_+=\{x=(x',x_N)\in \R^N:\ x_N>0\}$ stands for the half-space and $f$ is a locally 
Lipschitz nonlinearity. We completely characterize one-dimensional bounded solutions of this problem, and we prove 
among other things that if $u$ is a bounded solution with $\rho:=\sup_{\mathbb{R}^N}u$ verifying $f(\rho)=0$, then $u$ is 
necessarily one-dimensional. 
\end{abstract}

\section{Introduction}

In this paper we study existence and qualitative properties of positive, bounded solutions of the semi-linear 
nonlocal equation:
\begin{equation}\label{eq:Problema}\tag{$P_N$}
		\begin{cases}
			(-\Delta)^s u = f(u) &\text{ in } \mathbb{R}^N_+,\\
			\ \ u=0 &\text{ in } \mathbb{R}^N\setminus \mathbb{R}^N_+,
		\end{cases}
	\end{equation}
where $\mathbb{R}^N_+=\{x=(x',x_N)\in \mathbb{R}^N:\ x_N>0\}$ is the half-space and $f$ is a locally Lipschitz 
function. Here $(-\Delta)^s$ denotes the \emph{fractional laplacian}, which
is defined on smooth functions as
\begin{equation}\label{operador}
(-\Delta)^s u(x) = c(N,s) \int_{ \mathbb{R}^N} \frac{u(x)-u(y)}{|x-y|^{N+2s}} dy,
\end{equation}
where $c(N,s)$ is a normalization constant given by
\begin{equation}\label{eq-const-norm}
c(N,s)= 4^s s(1-s) \pi^{-\frac{N}{2}} \frac{\Gamma\left( s+\frac{N}{2}\right)}{\Gamma(2-s)}
\end{equation}
(cf. Lemma 5.1 in \cite{MR2754080}). The integral in \eqref{operador} 
has to be understood in the principal value sense.

Before stating our results, let us briefly discuss the known achievements for the local case $s=1$, which motivates our study. 
The more relevant references in the subject are a series of papers by Berestycki, Caffarelli and Nirenberg, 
\cite{MR1260436, MR1395408, MR1655510, MR1470317}, where qualitative properties of solutions of 
\begin{equation}\label{problema-local}
\left\{
\begin{array}{ll}
-\Delta u = f(u) & \hbox{in }\mathbb{R}^N_+,\\
\ \ u=0 & \hbox{on } \partial \mathbb{R}^N_+,
\end{array}
\right.
\end{equation}
were obtained. The two main properties analyzed there are the monotonicity of solutions of \eqref{problema-local} and 
their one-dimensional symmetry (sometimes called rigidity). In some of these papers, some more general unbounded 
domains were also considered.

With regard to monotonicity properties in the case $s=1$, the first known result in the half-space seems to be 
due to Dancer in \cite{MR1190345}, although monotonicity in some coercive epigraphs was shown before in \cite{MR688279}. 
The more general case where $f$ is a Lipschitz function and $f(0)\geq 0$ is solved in \cite{MR1395408, MR1655510}. It is shown 
there that all positive solutions $u$ of \eqref{problema-local}, not necessarily bounded, are monotone in the $x_N$ 
direction. The case $f(0)<0$ is more delicate, and nowadays still not completely solved. 
See \cite{MR3593525} for several achievements in $N=2$, and \cite{MR3550849} for some partial results in higher dimensions. 
The main reason is the existence of a one-dimensional, periodic solution of \eqref{problema-local} which is not strictly positive.

As for the symmetry of solutions of \eqref{problema-local}, it is only conjectured that all bounded solutions 
are necessarily one-dimensional; see \cite{MR1655510}. This conjecture was shown to be true when $N=2$ or when $N=3$ 
and $f(0)\ge 0$ in \cite{MR1655510}. In higher dimensions the only general result in this direction at the 
best of our knowledge is the one in \cite{MR1260436}, where it is proved that if $\rho:=\sup u$ verifies 
$f(\rho)\leq 0$, then $u$ is symmetric and one additionally has $f(\rho)=0$. 
A slightly more restrictive version of this result had been previously proved by Angenent in 
\cite{MR794096} and Cl{\'e}ment and Sweers in \cite{MR937538}.

\medskip  

Back to our nonlocal problem \eqref{eq:Problema}, the question of monotonicity for positive, bounded solutions 
has been addressed before in some works. We mention preliminary results obtained in \cite{MR3454619} and 
\cite{MR3311908} for special nonlinearities, and a fairly general recent result by the authors in \cite{MR3624935}, 
where it is shown that nonnegative bounded solutions of \eqref{eq:Problema} are increasing in the 
$x_N$ direction even in the more delicate case $f(0)<0$. The only additional requirement is that 
$f$ needs to be $C^1$. Let us also mention the paper \cite{2016arXiv160407755S}, where some monotonicity 
properties are obtained for some more general unbounded domains and some special nonlinearities.

Nevertheless, the question of symmetry for positive, bounded solutions of \eqref{eq:Problema} is 
far from being completely analyzed. We are only aware of Corollary 1.2 in \cite{2016arXiv160407755S}, 
where some special nonlinearities are dealt with.

\bigskip

Next we describe our main results. First, let us comment that with the exception of Section 2, 
we will be mainly dealing with classical solutions of \eqref{eq:Problema}. However, it can be seen with the use of 
the regularity theory developed in \cite{MR2270163,MR2494809,MR2781586} and bootstrapping arguments 
that bounded, viscosity solutions of (1.1) in the sense introduced in \cite{MR2494809} are automatically 
classical. See the Appendix for a definition of viscosity solution.

We begin by considering the one-dimensional version of problem 
\eqref{eq:Problema}, that is 
\begin{equation}\label{eq:Problema-unidim}\tag{$P_1$}
		\begin{cases}
			(-\Delta)^s u = f(u) &\text{ in } \mathbb{R}_+,\\
			\ \ u=0 &\text{ in } \mathbb{R}\setminus \mathbb{R}_+.
		\end{cases}
	\end{equation}
At the best of our knowledge, this problem is not very well understood at present. The fact 
that $N=1$ does not substantially simplify the expression of the operator $(-\Delta)^s$ 
seems to be responsible for this lack of knowledge. In spite of this, when the problem is 
posed in $\R$ and special solutions are taken into account, there has been some progress in 
\cite{MR3165278}, \cite{MR3280032}.

When $s=1$, however, the corresponding problem 
\begin{equation}\label{eq:prob-unid-local}
		\begin{cases}
			- u'' = f(u) &\text{ in } \mathbb{R}_+,\\
			\ \ u(0)=0 
		\end{cases}
	\end{equation}
has been extensively studied, and it is easy to see that there exists a bounded positive solution of 
\eqref{eq:prob-unid-local} if and only if $\rho=\|u\|_{L^\infty(\R_+)}$ verifies 
$f(\rho)=0$ and 
	\begin{align}
		\label{HF1}\tag{F} 
		& F(t)<F(\rho) \text{ for all } t\in [0,\rho),
	\end{align}
where $F$ is the primitive of $f$ vanishing at zero, $F(t)=\int_0^t f(\tau) d\tau$. Moreover, 
the solutions are increasing in $x$, and there exists a unique solution with a prescribed value of 
$\rho$. Thus problem \eqref{eq:prob-unid-local} admits as many solutions as zeros of 
$f$ verifying condition \eqref{HF1}. 
This is an immediate consequence of the existence of an energy for solutions of \eqref{eq:prob-unid-local}. 
As can be directly checked, if $u$ is a solution of $-u''=f(u)$ in $(0,+\infty)$, the function 
$$
E(x):= \frac{u'(x)^2}{2} + F(u(x)), \qquad x>0,
$$
is constant. It is also important to remark that the uniqueness of solutions for initial value problems 
associated to the equation plays also an important role in this characterization.

On the contrary, for the nonlocal problem \eqref{eq:Problema-unidim}, no energy is known to exist 
for the moment despiste the Hamiltonian identity obtained in \cite{MR3165278, MR2177165} for layer solutions using the extension tool \cite{MR2354493}, and of course initial value problems have no sense in its context. Thus existence and 
uniqueness of solutions and their monotonicity have to be shown in an alternative way. 

Indeed, we will prove that problem \eqref{eq:Problema-unidim} possesses the same features as the local 
version, by constructing solutions in a different way. In addition, we will also obtain a nonlocal energy 
which can be used to show the uniqueness of solutions with a prescribed maximum. We strongly believe that this 
energy could be useful in other one-dimensional problems.

Before stating our main result, we remark that solutions of \eqref{eq:Problema-unidim} are expected 
to have a singular derivative at $x=0$, and we need to consider the `fractional derivative'
\begin{equation}\label{def-elecero}
\ell_0:=\lim_{x\to 0^+} \frac{u(x)}{x^s}.
\end{equation}
The existence of this limit for solutions of \eqref{eq:Problema-unidim} is consequence of the 
regularity results in \cite{MR3168912}.

We will establish now the main results of the work:
\begin{theorem}\label{theorem:unidim}
Assume $f$ is locally Lipschitz and $\rho>0$ is such that $f(\rho)=0$ and condition \eqref{HF1} is verified. Then 
there exists a unique positive solution $u$ of \eqref{eq:Problema-unidim} with the property
$$
\|u \|_{L^\infty(\R)}=\rho.
$$
Moreover, $u$ is strictly increasing and $\ell_0$ in \eqref{def-elecero} is given by
\begin{equation}\label{eq:elecero}
\ell_0= \frac{(2F(\rho))^\frac{1}{2}}{\Gamma(1+s)}.
\end{equation}
Finally, all bounded positive solutions of \eqref{eq:Problema-unidim} are of the above form.
\end{theorem}

For every positive $\rho$ verifying \eqref{HF1} we denote the unique positive solution given 
by Theorem \ref{theorem:unidim} by $u_{\rho}$. 

\medskip

Once solutions of the one-dimensional problem are completely understood, we expect them to give 
rise to special solutions of \eqref{eq:Problema}. While in the local case $s=1$ this is immediate, 
it is not completely straightforward when $s\in (0,1)$, due to the presence of a constant in 
the definition of $(-\Delta)^s$ which depends on the dimension $N$. We are unaware if this fact 
is already present somewhere in the literature, but we include a proof for completeness. That is, we have the following.

\begin{proposition}\label{theorem:exist-half}
Under the conditions of Theorem \ref{theorem:unidim}, let $u_{\rho}$ be the bounded positive solution 
of the one-dimensional problem \eqref{eq:Problema-unidim}. Then the function 
\begin{equation}\label{eq-urho}
u(x)= u_{\rho}(x_N), \qquad x\in \R^N
\end{equation}
is a bounded positive solution of \eqref{eq:Problema}. Conversely, if $u$ is a bounded solution of 
\eqref{eq:Problema} which depends only on $x_N$, then \eqref{eq-urho} holds for some $\rho>0$ such that 
$f(\rho)=0$ and \eqref{HF1} is verified.
\end{proposition}

\bigskip

In the light of Proposition \ref{theorem:exist-half}, it is natural to ask as in the local case 
whether all bounded, positive solutions of \eqref{eq:Problema} come from solutions of 
\eqref{eq:Problema-unidim}. Thus we pose the following

\medskip

\begin{quote}
{\bf Conjecture:} assume $f$ is locally Lipschitz and let $u$ be a bounded positive solution of 
\eqref{eq:Problema}. Then $u$ is one-dimensional.
\end{quote}

\medskip

We are unable to prove this conjecture in its full generality, but we will address some 
particular instances which are generalizations of some known facts in the local case. We begin by 
considering the case where the maximum of $u$ is a zero of $f$, as in \cite{MR1260436}. To be 
more precise, it was assumed there that $f(\|u\|_{L^\infty(\R^N)})\le 0$, but it is easily seen 
that this condition is equivalent to $f(\|u\|_{L^\infty(\R^N)})= 0$. Then we have the next

\begin{theorem}\label{theorem:main}
		Assume $f$ is locally Lipschitz and let $u$ be a bounded positive solution of \eqref{eq:Problema}. 
		Suppose in addition that $\rho=\|u\|_{L^\infty(\R^N)}$ verifies $f(\rho)=0$. Then $f$ verifies 
		\eqref{HF1} and $u$ is one-dimensional. More precisely, 
		$$
		u(x)= u_{\rho}(x_N), \quad x\in \R^N.
		$$
\end{theorem}

\bigskip

The proof of Theorem \ref{theorem:main} ultimately relies in obtaining good lower bounds for the 
solutions $u$ which allow us to construct a one-dimensional solution below it. It is precisely in 
this step when the condition $f(\rho)=0$ is important. When this condition is not assumed, we 
can still say something by placing some restriction on the behaviour of $f$ at zero. The usual 
condition 
\begin{equation}\label{hip-cero}
\liminf_{t\to 0^+} \frac{f(t)}{t}>0
\end{equation}
has been considered at several places in the literature of local problems with the same objective 
(cf. for instance \cite{MR1470317}). 

A generalization of the results in \cite{MR1470317} has 
been recently obtained in \cite{2016arXiv160407755S}. When it comes to the half-space, it was shown 
there that if $f$ is a function that has a unique positive zero $\rho$, that initially it does not have to be the supremum of the solution, that verifies \eqref{hip-cero} and is negative for values 
larger than $\rho$ and nonincreasing near $\rho$, then every positive, bounded solution of \eqref{eq:Problema} 
is one-dimensional. We improve Corollary 1.2 there, in the sense that we do not require the monotonicity 
condition on $f$ and we show moreover that the solution is unique.

\begin{theorem}\label{theorem:valdinoci}
Assume $f$ is locally Lipschitz and verifies $f>0$ in $(0,\rho)$, $f<0$ in $(\rho,+\infty)$ and 
\eqref{hip-cero}. Then the unique bounded, positive solution of \eqref{eq:Problema} 
is 
$$
u(x)= u_{\rho}(x_N), \quad x\in \R^N,
$$
where $u_{\rho}$ is the unique solution of $(P_1)$ with $\|u\|_{L^{\infty}(\mathbb{R}^{N})}=\rho$ given by Theorem \ref{theorem:unidim}.
\end{theorem}

\bigskip

As a corollary of Theorem \ref{theorem:valdinoci} we obtain a Liouville theorem for a  
particular class of nonlinearities. 

\begin{corollary}\label{Liouville}
Assume $f$ is locally Lipschitz and verifies $f>0$ in $(0,+\infty)$ and \eqref{hip-cero}. 
Then problem \eqref{eq:Problema} does not admit any positive, bounded solution.
\end{corollary}

\bigskip

To conclude the introduction we will briefly comment on our methods of proof. With regard 
to the one-dimensional problem \eqref{eq:Problema-unidim}, the existence of solutions follows 
by means of sub and supersolutions. It is worthy of mention that precise subsolutions have 
to be constructed in order to ensure that the solutions so obtained have the desired 
$L^\infty$ norm. These subsolutions are shown to exist with an adaptation of the results in 
\cite{MR937538}. As for uniqueness, it is obtained thanks to Hopf's Lemma and the 
characterization \eqref{eq:elecero}. 
This characterization follows because of our nonlocal energy, furnished by Theorem \ref{thm-pral-energia} 
below. The energy is obtained by direct integration of the expression $u'(x)(-\Delta)^s u(x)$, with a 
careful analysis of all the appearing terms. It is to be noted that the same expression can be 
obtained with the results in \cite{MR3211861}, which however need the restriction $f(u)\in L^1$. 
This could not hold in general. 

As for the rest of our theorems, most of them follow with the use of the well-known sliding method, 
see \cite{MR1159383}. However, some additional care is needed because the subsolutions we slide 
do not have a compact support, which is the usual situation. The method of sub and supersolutions, 
providing with a maximal solution in each case is the other essential tool in our approach. 

\medskip

The rest of the paper is organized as follows: Section 2 is dedicated to the existence of 
solutions for problems \eqref{eq:Problema} and \eqref{eq:Problema-unidim}. In Section 3, 
we obtain our nonlocal energy and use it to prove the uniqueness of solutions of 
\eqref{eq:Problema-unidim}. Section 4 is devoted to the proof of our main results, and an 
Appendix is included dealing with the method of sub and supersolutions.

\bigskip

\section{Existence of solutions}\label{existence}
\setcounter{equation}{0}
	
	In this section we are concerned with the existence of positive solutions 
	of the problem 
	\begin{equation}\label{eq-unidim}\tag{$P_1$}
	\begin{cases}
			(-\Delta)^s u = f(u) &\text{ in } \mathbb{R}_+,\\
			\ \ u=0 &\text{ in } \mathbb{R}\setminus \mathbb{R}_+.\\
	\end{cases}
	\end{equation}
	More precisely, if the function $f$ is locally Lipschitz and $\rho>0$ is such that $f(\rho)=0$ and 
	\eqref{HF1} is satisfied, then 
	we will show that there exists a positive, viscosity solution of \eqref{eq-unidim} which is 
	increasing in $x$ and verifies in addition $\lim_{x\to +\infty} u(x)=\rho$. Recall that viscosity 
	solutions are automatically classical.
	
  To simplify the notation, throughout this section we will omit 
	the normalization constant $c(N,s)$ in the definition of the fractional laplacian. 
	
	\subsection{Existence of solutions in a ball}
	
	Although we will primarily deal with the one-dimensional problem \eqref{eq-unidim}, in the procedure 
	we need to consider several related problems which are posed in finite domains. For its use in Section 
	\ref{sec-main}, we will analyze the $N$-dimensional problem 
		\begin{equation}\label{eq:PBD}
			\begin{cases}
				(-\Delta)^s u = f(u) &\text{ in } 
				B_R,\\
				\ \ u=0 &\text{ in } B_R^{c}=\mathbb{R}^N\setminus
				 B_R,
			\end{cases}
		\end{equation} 
	where $B_R\subset\mathbb{R}^N$, $N\geq 1$, stands for the ball of radius $R$ centered at the origin. 
	However, all the results in this section are directly generalized to problems where the underlying 
	domain is a dilation of a fixed one.

In general, there is no hope that problem \eqref{eq:PBD} admits nonnegative solutions. This is the reason 
why we are imposed in a first stage the additional assumption $f(0)\ge0$. 
	
	\begin{lemma}\label{lemma:EBD1} 
		Assume $f$ is locally Lipschitz in $\R$ and $\rho>0$ is such that $f(\rho)=0$ and \eqref{HF1} is satisfied, 
		together with $f(0)\ge 0$. Then for every 
		$\varepsilon>0$ there exists a positive number $R_{0}=R_0(\varepsilon)$ such that for $R\ge R_0$, problem 
		\eqref{eq:PBD} admits a positive viscosity solution $u_R\in C^s(\mathbb{R}^N)$, verifying in addition
			\begin{equation}\label{eq:EBD11}
			\rho-\varepsilon\le 
			\|u_R\|_{\scriptstyle L^\infty(B_R)} < \rho.
		\end{equation}	
	\end{lemma}
	
	\begin{proof}
	The proof is an adaptation of that of Lemma 2.1 in \cite{MR937538}, where the local case
	 $s=1$ was analyzed. We split it in 
	two steps.
		
		\medskip
		
		\noindent {\it Step 1.} First we  show that for every $R>0$ 
		there exists a viscosity solution $u_R\in C^{s}(\mathbb{R}^N)$ of  \eqref{eq:PBD} such that   
		$0\le u_R\le \rho$ in $B_R$. For this aim, we define an auxiliary function $\widetilde{f}$ 
		by setting $\widetilde{f}(t)=f(t)$ in $[0,\rho]$, 
		$$
		\widetilde{f}(t)= 0 \quad \text{ for }t>\rho
		$$
		and extend it to negative values by means of 
	 	$$
		\widetilde{f}(t)= 2f(0)- \widetilde{f}(-t) \quad\text{ if } t<0.
		$$
		Observe that the function $\widetilde{f}$ is bounded in $\R$, and $\widetilde{f}(t)-f(0)$ is odd 
		by its very definition. Denote
		\begin{equation}\label{eq:laF}
			\widetilde{F}(t)\coloneqq\int^t_0 \widetilde{f}(s)  ds.
		\end{equation}
	 Next, in the Hilbert space
	\[
		\widetilde{H}(B_R)
		\coloneqq\{u\in H^s(\mathbb{R}^N)
		\colon u=0 \hbox{ a.e. in } B_R^{c}\}
	\]
we define the following functional:
		\[
			J(v)\coloneqq
			\dfrac12\iint_{\R^{2N}}
			\dfrac{|v(x)-v(y)|^2}{|x-y|^{N+2s}}dxdy-
			\int_{B_R} \widetilde{F}(v) dx,
		\]
		(we refer the reader to \cite{MR2879266} or 
		\cite{MR3002745} for a definition of $H^s(\R^N)$ and the use of variational methods for boundary 
		value problems involving the fractional laplacian).

		Observe that $J$ is sequentially weakly lower semicontinuous and the boundedness of $\widetilde{f}$ 
		implies that it is also coercive in $\widetilde{H}^s(B_R)$. Thus it 
		possesses a global minimizer $u_R\in \widetilde{H}^s(B_R)$. We claim that 
		indeed $u_R$ can be chosen to verify 
		\begin{equation}\label{eq-cotas}
		0\le u_R\le \rho. 
		\end{equation}
		To prove the first inequality in \eqref{eq-cotas} we will show that for every $v\in \widetilde{H}^s(B_R)$ 
		we have
		\begin{equation}\label{eq-nonneg}
			J(|v|)\le J(v)
		\end{equation}
		which clearly implies that $u_R$ can be taken to be nonnegative. To show 
		\eqref{eq-nonneg} it is enough to notice that, since $\widetilde{f}(t)-f(0)$ is an odd function, 
		then its primitive $\widetilde{F}(t)-f(0)t$ is even, so that for $t>0$:
		$$
		\widetilde{F}(-t) = \widetilde{F}(t)-2f(0)t \le \widetilde{F}(t),
		$$
		owing to our extra condition $f(0)\ge 0$. This immediately yields $\widetilde{F}(t) \le \widetilde{F}(|t|)$ 
		for $t\in \R$. Since it is also well-known that 
		$$
		||v(x)|-|v(y)||\le |v(x)-v(y)| \quad \text{ for every }
			x,y\in\mathbb{R}^N,
		$$
		then \eqref{eq-nonneg} follows. 
		
		To show the second inequality in \eqref{eq-cotas} we define 		
		$w(x)=\min\{u_R(x),\rho\}$. Observing that $F(t)=F(\rho)$ whenever $t>\rho$ and that 
		\[
			|w(x)-w(y)|\le|u_R(x)-u_R(y)| \quad \text{ for every }
			x,y\in\mathbb{R}^N, 
		\]
		it directly follows that $J(w)\le J(u_R)$. Thus replacing $u_R$ by $w$, we may always assume that 
		the second inequality in \eqref{eq-cotas} holds.
		
		By a standard argument, a minimizer of $J$ in $\widetilde{H}^s(B_R)$ is a weak solution of 
		\eqref{eq:PBD}. In addition, since $f(u_R)\in L^\infty(B_R),$ we deduce using Proposition 1.1 in 
		\cite{MR3168912} that $u_R\in C^s(\mathbb{R}^N)$. Moreover, since then the right-hand side of 
		\eqref{eq:PBD} is a continuous function, then $u_R$ is a viscosity solution of \eqref{eq:PBD}  
		(cf. Remark 2.11 in \cite{MR3168912} or Remark 6 in \cite{MR3161511}).		
			
		\bigskip
		
		\noindent {\it Step 2.} We  prove that for any $\varepsilon>0$ there is a positive number 
		$R_{0}=R_{0}(\varepsilon)$ such that $u_R$ is positive in $B_R$ and \eqref{eq:EBD11} holds if $R>R_{0}.$ 
		
		We begin by observing that the scaled function $w_R(x)=u_R(Rx)$ is a 
		minimizer of
		\[
			J_R(v)\coloneqq
			\dfrac12\iint_{\R^{2N}}
			\dfrac{|v(x)-v(y)|^2}{|x-y|^{N+2s}}dxdy-R^{2s}
			\int_{B_1} \widetilde{F}(v) dx
		\]
		in $\widetilde{H}^s(B_1).$ 
		As a first step in proving \eqref{eq:EBD11}, we will show that
		given $\varepsilon>0$ there is a positive number $R_{0}$ 
		such that
		\begin{equation}\label{eq-supnorm}
			\rho-\varepsilon\le 
			\|w_R\|_{\scriptstyle L^\infty(B_1)}
			\le \rho \quad \text{ for every } R>R_{0}.
		\end{equation}
		Suppose that \eqref{eq-supnorm} does not hold. Then there exist $\varepsilon>0$ and a sequence 
		$R_n\to +\infty$ such that $\|w_n\|_{L^\infty(B_1)}<\rho-\varepsilon,$ where
		$w_n=w_{R_n}.$ 
		Define
		\begin{align*}
			\alpha &=\min\left\{F(\rho)-F(r) \colon 0\le r
			\le \rho-\varepsilon \right\},\\
			\beta 
			&=\max\left\{F(\rho)-F(r)\colon 0\le r
			\le \rho \right\}.
		\end{align*}
		Since, by \eqref{HF1}, $\alpha>0$, we can choose $\delta>0$ small enough to have
		\begin{equation}\label{alph}
			|B_1^\delta|\beta<|B_1|\alpha,
		\end{equation} 
		where 
		$B_1^\delta
		=\{x\in B_1 \colon {\rm dist}(x,\partial B_1)<\delta\}$ 
		and $|\cdot|$ stands for the Lebesgue measure. 
		
		\medskip
		
		We next choose a function $w\in C_0^\infty(B_1)$ 
		satisfying $0\le w(x)\le \rho$ in $B_1^\delta$ and 
		$w\equiv\rho$ in $B_1\setminus B_1^\delta$. Then for a positive constant $C$:
		$$
		\begin{array}{rl}
			J_{R_n}(w)-J_{R_n}(w_n) \hspace{-2mm} & \le \dfrac12
			\displaystyle \iint_{\R^{2N}} \dfrac{|w(x)-w(y)|^2}{|x-y|^{N+2s}}dxdy
			-R^{2s}_n \int_{B_1} (F(w)-F(w_n))\\[.75pc]
			& =C - R^{2s}_n \displaystyle \left(
			\int_{B_1} (F(\rho)-F(w_n)) - \int_{B_1^\delta} (F(\rho)-F(w)) \right)\\[1pc]
			& \le C- R_n^{2s} (\alpha |B_1|-\beta |B_1^\delta|)<0
		\end{array} 
		$$
		for large $n$, thanks to \eqref{alph}. This is a contradiction with the fact that 
		$w_n$ is a minimizer of $J_{R_n}$, which shows that \eqref{eq-supnorm} must be true.
		
		\smallskip
				
		Coming back to the functions $u_R$, we see that 
		\eqref{eq:EBD11} holds except for the strict inequality. However, since $u_R$ is a viscosity 
		solution of \eqref{eq:PBD}, $f$ is locally 
		Lipschitz and $f(\rho)=0$, it is standard by the strong maximum principle that $u_R<\rho$ 
		in $B_R$. Observe that the strong maximum principle also implies that $u_R>0$ in $B_R$, 
		concluding the proof of the lemma. 
		\end{proof}
	
	It is now the turn to remove the extra assumption $f(0)\ge 0.$ As observed before, without 
	this hypothesis we can not guarantee the existence of a positive solution. However, it will be 
	enough for our purposes in the near future to obtain slightly negative solutions. To this aim, we 
	will redefine the function $f$ for negative values when $f(0)<0$. Observe that 
	for small enough positive $\delta$ we have 
	\begin{equation}\label{eq:fdelta}
				\dfrac{f(0)}2\delta +F(\rho) >0.
	\end{equation}
	We define the function $f_\delta$ in $[-\delta,\rho]$ by setting
	\begin{equation}\label{fdelta}
		f_\delta(t)\coloneqq
		\begin{cases}
			\dfrac{f(0)}{\delta}(t+\delta) &\text{ if } t\in[-\delta,0),\\[0.25pc]
			f(t) &\text{ if } t\in[0,\rho].
		\end{cases}	 
	\end{equation}
	In the case $f(0)\ge0,$ we simply take $\delta=0$ and $f_0=f.$
	We now consider a slight variant of problem \eqref{eq:PBD} with $f$ replaced by $f_\delta$ 
	and a negative datum outside $B_R$, namely:
	\begin{equation}\label{eq:PBD2}
			\begin{cases}
				(-\Delta)^s u = f_\delta(u) &\text{ in } B_R,\\
				\ \ u=-\delta &\text{ in } \mathbb{R}^N\setminus B_R.
			\end{cases}
		\end{equation} 
	Note that the function $g_\delta(t)=f_\delta(t-\delta)$ is locally Lipschitz and satisfies 
	\eqref{HF1} with $\rho$ replaced by $\rho+\delta$. Hence we can apply Lemma \ref{lemma:EBD1} to obtain that, for 
	every $\varepsilon>0$, there exists $R_0>0$ such that for $R\ge R_0$, problem \eqref{eq:PBD} with $f$ replaced 
	by $g_\delta$ admits a positive viscosity solution $w_{\delta,R}\in C^s(\mathbb{R}^N)$ verifying 
	\[
		\rho+\delta-\varepsilon\le 
		\|w_{\delta,R}\|_{\scriptstyle L^\infty(B_R)}< \rho+\delta.
	\]
	Setting $u_{\delta,R}=w_{\delta,R}-\delta,$ we get the 
	following result:
		
	\begin{lemma}\label{corollary:EBD}		
			Assume $f$ is locally Lipschitz in $\R$ and $\rho>0$ is such that \eqref{HF1} is verified. 
			If $\delta>0$ is small enough so that \eqref{eq:fdelta} holds, then for every 
		$\varepsilon>0$ there exists a positive number $R_{0}=R_0(\varepsilon)$ such that for $R\ge R_0$, problem 
		\eqref{eq:PBD2} admits a viscosity solution $u_{\delta,R} \in C^s(\mathbb{R}^N)$, verifying 
		$u_{\delta,R}>-\delta$ in $B_R$ and 
			$$
			\rho-\varepsilon\le 
			\|u_{\delta,R}\|_{\scriptstyle L^\infty(B_R)} < \rho.
			$$
		\end{lemma}
	
	\bigskip
	
	We next observe that, thanks to Theorem \ref{theorem:SubSup} in the Appendix, whenever a 
	viscosity solution $u$ of \eqref{eq:PBD2} exists with the property $u\le \rho$ in $\R^N$, then 
	a maximal viscosity solution $\widetilde{u}$ of the same problem and with the same property 
	also exists. Here and in what follows, by ``maximal" we mean maximal with respect to the supersolution 
	$\bar u=\rho$, that is, if $v$ is any viscosity solution of \eqref{eq:PBD2} with $v\le \rho$ in $\R^N$ 
	then we have $v\le \widetilde{u}$ in $\R^N$. 
	
	On the other hand, by Theorem 1.1 in \cite{MR3189604}, every positive solution of \eqref{eq:PBD} with 
	$f$ replaced by $g_\delta$ is radially symmetric and radially decreasing. Thus we immediately have:

	\begin{lemma}\label{lemma:EBDM}
		Under the same assumptions as in Lemma \ref{corollary:EBD}, for every 
		$\varepsilon>0$ there exists a positive number $R_{0}=R_0(\varepsilon)$ such that for $R\ge R_0$, problem 
		\eqref{eq:PBD2} admits a maximal viscosity solution $\widetilde{u}_{\delta,R} \in C^s(\mathbb{R}^N)$, verifying 
		$\widetilde{u}_{\delta,R}>-\delta$ in $B_R$ and 
			$$
			\rho-\varepsilon\le 
			\|\widetilde{u}_{\delta,R}\|_{\scriptstyle L^\infty(B_R)} < \rho.
			$$
		Moreover, $\widetilde{u}_{\delta,R}$ is radially symmetric and radially decreasing.
	\end{lemma}
	
	\medskip
	
	\begin{remark}\label{remark:Order}{\rm 
		Let $R_1<R_2$ and denote by $\widetilde{u}_{\delta,1}, \widetilde{u}_{\delta,{2}}$ 
		the maximal viscosity	solutions of \eqref{eq:PBD2} with 
		$R=R_1$ and $R=R_2,$ respectively. 
		Then $w(x)=\max\{\widetilde{u}_{\delta,1}(x),
		\widetilde{u}_{\delta,2}(x)\}$
		is a viscosity subsolution of \eqref{eq:PBD2} with 
		$R=R_2$. By Theorem \ref{theorem:SubSup} in the Appendix, there exists a solution in 
		the ordered interval $[w,\rho]$, and therefore the maximal solution lies in that interval, 
		that is $w\le \widetilde{u}_{\delta,2}$ in $B_{R_2}$, in fact in $\R^N$. Hence
		\[
			\widetilde{u}_{\delta,1}\le \widetilde{u}_{\delta,2} 
			\text{ in } \R^N.
		\]
		With a similar argument, and taking into account that $f_{\delta}$ is decreasing 
		with respect to $\delta$ we can deduce that, if $\delta_1<\delta_2$ then
		\[
			\widetilde{u}_{1,R}\ge\widetilde{u}_{2,R} 
			\text{ in } \R^N
		\]
		where now $\widetilde{u}_{1,R}$ and 
		$\widetilde{u}_{2,R}$ are 
		the maximal viscosity
		solutions of \eqref{eq:PBD2} with $\delta=\delta_1$
		and $\delta=\delta_2,$ respectively.
	}\end{remark}

	
	\medskip
	
	\subsection{Existence of solutions in $\mathbb{R}_+$}
	
	Next we consider again the one-dimensional problem \eqref{eq-unidim}. 
	The purpose of this subsection is to obtain the following existence result:

	\begin{theorem}\label{theorem:MaximalSolR}
		Assume $f$ is locally Lipschitz in $\R$ and $\rho>0$ is such that $f(\rho)=0$ and 
		\eqref{HF1} is verified. Then problem 
		\eqref{eq-unidim} admits a maximal viscosity solution $u\in C^s(\R)$, which is positive and verifies 
		$$
		\| u\|_{L^\infty(\R)}=\rho.
		$$
		In addition, $u$ is strictly increasing for $x>0$ and
		$$
		\lim_{x\to +\infty} u(x)=\rho.
		$$
	\end{theorem}
	
	\bigskip
	
	The way to achieve existence of solutions of \eqref{eq-unidim} is to establish it first for 
	a $\delta-$variation of this problem, that is, 
		\begin{equation}\label{eq:PBDeltaR}
	\begin{cases}
			(-\Delta)^s u = f_\delta(u) &\text{ in } \mathbb{R}_+,\\
			\ \ u=0 &\text{ in } \mathbb{R}\setminus \mathbb{R}_+,\\
	\end{cases}
	\end{equation}
	where $f_\delta$ is given by \eqref{eq:fdelta}, and then pass to 
	the limit as $\delta\to 0^+$. Let us recall that we take $\delta=0$ and $f_\delta=f$ when $f(0)\ge 0$.
	
	\begin{lemma}\label{lemma:ExistenceDelta}
		With the same assumptions as in Theorem \ref{theorem:MaximalSolR}, 
		there exists a viscosity solution $u_{\delta}$ of \eqref{eq:PBDeltaR} such that 
		$-\delta< u_{\delta}<\rho$ in $\mathbb{R}_+$ and  
		$\|u_\delta\|_{\scriptstyle L^{\infty}{(\mathbb{R})}}=\rho.$ 
	\end{lemma}
	
	\begin{proof}
	Fix $\varepsilon>0$. By Lemma \ref{lemma:EBDM}, there exists $R>0$ such that 
	problem \eqref{eq:PBD2} with $N=1$ and $B_R=(0,2R)$ admits a maximal viscosity 
	solution $u_R$ which verifies $\rho-\varepsilon \le \| u_R\|_{L^\infty(\R)} <\rho$.
	
	However, it is easily seen that the function $u_R$ is a subsolution of 
	problem \eqref{eq:PBDeltaR}. Thus by Theorem \ref{theorem:SubSupV} in the 
	Appendix (see also Remark \ref{rem-franja}), there exists a maximal solution $u_\delta$ of 
	\eqref{eq:PBDeltaR} relative to $\rho$, which verifies 
	$\rho-\varepsilon \le \| u_\delta\|_{L^\infty(\R)} <\rho$. 
	Since $u_\delta$ does not depend on $\varepsilon$, 
	it immediately follows that 
	$$
	\| u_\delta\|_{L^\infty(\R)} = \rho.
	$$
		Finally, since $f_\delta(-\delta)=f_\delta(\rho)=0$ and $f$ is locally Lipschitz, 
		we deduce from the strong maximum principle that $-\delta<u_\delta(x)<\rho$ in
		$\mathbb{R}_+.$  
	\end{proof}

	\bigskip

	\begin{proof}[Proof of Theorem \ref{theorem:MaximalSolR}]
	Remember that, when $f(0)\ge 0$ we are simply choosing $\delta =0$, so that there 
	exists a solution of \eqref{eq-unidim} by Lemma \ref{lemma:ExistenceDelta}. Therefore, 
	regarding existence, only the case $f(0)<0$ needs to be dealt with. 
			
	By the second part of Remark \ref{remark:Order}, we have that 
	if $\delta_1<\delta_2$ then $u_{\delta_1}\ge u_{\delta_2}$
	in $\mathbb{R}.$ Therefore
	\[
		v(x)\coloneqq \lim_{\delta\to 0^+} u_{\delta}(x)
		=\sup\left\{u_\delta(x)\colon \delta>0 \right\}.
	\]
	Observe that $0\le v\le\rho$ in $B_R$  and $\|v\|_{\scriptstyle L^\infty(\mathbb{R})}=\rho.$ 
	We next prove that $v$ is a solution of \eqref{eq-unidim}. 
	
	Choose $\delta_n\to 0^+$ and let $u_n=u_{\delta_n}$. First, observe that  for any $n\in\mathbb{N}$
		\[
			\|u_n\|_{\scriptstyle L^\infty(\mathbb{R})}\le \rho
			\quad\text{ and }\quad
			\|f_\delta(u_n)\|_{\scriptstyle L^\infty(\mathbb{R}_+)}
			\le\|f\|_{\scriptstyle L^\infty(0,\rho)}.
		\]   
		With the use of standard interior regularity (see for instance Theorem 12.1 in 
		\cite{MR2494809}) we can obtain appropriate interior bounds for the H\"older 
		norms of the solutions. More precisely, for every $b>a>0$ we have 
		\[
			\|u_{n}\|_{C^s[a,b]} \le 
			C\left(\|f_\delta(u_n)\|_{\scriptstyle L^\infty(\mathbb{R}_+)}
			+\|u_n\|_{\scriptstyle L^\infty(\mathbb{R})}\right)
			\le C\left(\|f\|_{\scriptstyle L^\infty(0,\rho)}+\rho
			\right)
		\]
		for some positive constant $C=C(a,b)$. Hence,
		we can conclude that $\{u_n\}_{n\in\mathbb{N}}$ is an  
		equicontinuous and uniformly bounded sequence. It follows that 
		$u_n\to v$ locally uniformly in $\mathbb{R}_+$, so that 
		$v \in C(\R\setminus \{0\})$ and $v=0$ in $(-\infty,0)$. Observe that with 
		this procedure it is not immediate that $v(0)=0$ and $v$ is continuous at zero. 
		However, we can argue as in Theorem \ref{theorem:SubSupV} in the Appendix to 
		obtain that actually $v \in C(\R)$ and $v(0)=0$.
		
		We can now use Lemma 4.7 in \cite{MR2494809}, which shows that $v$ is indeed a viscosity 
		solution of \eqref{eq-unidim} with $0 \le v<\rho.$ 	
	By Theorem \ref{theorem:SubSupV} in the Appendix, there exists a maximal viscosity solution $u\in C(\R)$ 
	of \eqref{eq-unidim}, which of course verifies $0\le u<\rho$ and $\|u\|_{L^\infty(\R)}=\rho$. 
	
	Thus to conclude the proof, only the strict monotonicity of $u$ in $\R_+$ remains to be shown, 
	since it will imply $u>0$ in $(0,+\infty)$. We mention in passing that the monotonicity of 
	$u$ is a consequence of Lemma \ref{lema-monotonia} below, but we are providing an independent 
	proof of this fact. 
	
	Choose $\lambda>0$ and consider the function $v_\delta(x)=u_\delta (x-\lambda)$. It is easily seen that 
	$v_\delta$ is a subsolution of \eqref{eq:PBDeltaR}. By Theorem \ref{theorem:SubSupV} in the Appendix, there exists a 
	solution $w_\delta$ of \eqref{eq:PBDeltaR} verifying $v_\delta \le w_\delta$ in $\R$. Arguing exactly as 
	in the first part of the proof, we can show that $w_\delta \to w$ locally uniformly in 
	$(0,+\infty)$, where $w\in C(\R)$ is a positive viscosity solution of \eqref{eq-unidim}. It follows 
	that 
	$$
	u(x-\lambda )\le w(x) \le u(x) \qquad \hbox{in } \R,
	$$
	since $u$ is the maximal solution. This shows that $u$ is monotone. Moreover, arguing as in Step 3 in 
	the proof of Theorem 1 in \cite{MR3624935}, we can show that $u'>0$ in $(0,+\infty)$, 
	so that $u$ is strictly monotone. The proof is concluded.
	\end{proof}
	
\bigskip	
	

	\section{Uniqueness}
	\setcounter{equation}{0}
	
		Our main objective in this section is the uniqueness of positive solutions of 
		the one-dimensional problem		
		$$
		\begin{cases}
			(-\Delta)^s u = f(u) &\text{ in } \mathbb{R}_+,\\
			\ \ u=0 &\text{ in } \mathbb{R}\setminus \mathbb{R}_+.
		\end{cases}
		\leqno{\eqref{eq:Problema-unidim}}
	$$
	In the procedure of proving this uniqueness, we will obtain 
		a nonlocal energy for the problem which we believe is interesting in its 
		own right, and could be further exploited to analyze other related one-dimensional 
		problems.
		
	\medskip
	
	\subsection{A nonlocal energy for one-dimensional solutions}

		The following is the main result of this subsection:
	\begin{theorem}\label{thm-pral-energia}
	Assume $f$ is locally Lipschitz and let $u$ be a positive, bounded solution of 
	\eqref{eq-unidim}. Then $u$ is strictly monotone in $(0,+\infty)$. Moreover, for every $a>0$ 
	we have
	$$
	\begin{array}{rl}
	F(u(a)) \hspace{-3mm} & - \ds \frac{c(1,s)}{2} \left( \int_{-\infty}^{+\infty} \frac{(u(a)-u(y))^2}{|a-y|^{1+2s}} dy 
	-(1+2s)\hspace{-1mm} \int_a^{+\infty} \hspace{-2mm} \int_{-\infty}^a \frac{(u(x)-u(y))^2}{|x-y|^{2+2s}} dy dx\right)\\[1.25pc]
	& =F(\rho),
	\end{array}
	$$
	where $\rho=\lim_{x\to +\infty} u(x)$. 
	In addition, if $\ell_0$ is given in \eqref{def-elecero}, then 
	$$
	\ell_0= \frac{(2F(\rho))^\frac{1}{2}}{\Gamma(1+s)}.
	$$
	\end{theorem}
	
	\bigskip
	
	\begin{remark}{\rm
	It can be seen with a little effort that the energy given by Theorem \ref{thm-pral-energia} 
	converges, as $s\to 1^-$, to the usual one for the local problem $E(x):=u'(x)^2/2+F(u(x))$.
	}\end{remark}
	
	The proof of Theorem \ref{thm-pral-energia} will be split in several lemmas for convenience. We begin by 
	showing the monotonicity of solutions of \eqref{eq:Problema-unidim}. We remark that the main result in 
	\cite{MR3624935} could be easily modified to include the case $N=1$. If we adapted the proof presented in this work we notice that the additional hypothesis $f\in C^1$, required there to obtain the monotonicity property of the solutions, will be not needed in the simpler situation of dimension one. However we give here an alternative proof that avoids the introduction of the notation established in \cite{MR3624935} regarding Green's function in half-spaces.
	
	\begin{lemma}\label{lema-monotonia}
		Assume $f$ is locally Lipschitz and let $u$ be a positive, bounded solution of 
		\eqref{eq-unidim}. Then $u'>0$ in $(0,+\infty)$.
	\end{lemma}
	
	\begin{proof}[Sketch of proof]
	The proof follows with the use of the moving planes method. 
	We borrow the notation from \cite{MR3624935}, which is for the most part standard.
	For $\lambda>0$, let 
	$$
	\begin{array}{l}
	\Sigma_\lambda :=(0,\lambda)\\[.5pc]
	x^\lambda:=2\lambda-x \ \hbox{(the reflection of }x \hbox{ with respect to the point } \lambda)\\[0.5pc]
	w_\lambda(x)= u(x^\lambda)-u(x), \quad x\in \R\\[0.5pc]
	D_\lambda=\{x\in \Sigma_\lambda: \ w_\lambda(x)<0\}\\[0.5pc]
	v_\lambda = w_\lambda \chi_{D_\lambda}.	
	\end{array}
	$$
	Observe that by Lemma 5 in \cite{MR3624935} we obtain $(-\Delta)^s v_\lambda \ge L v_\lambda$ in 
	the viscosity sense in $D_\lambda$, while $v_\lambda=0$ outside $D_\lambda$. Here $L$ stands for the Lipschitz 
	constant of $f$ in the interval $[0,\|u\|_{L^\infty(\R)}$]. 
	
	As a consequence of the maximum principle in narrow domains (which follows for instance 
	from Theorem 2.4 in \cite{MR3311908}) we deduce that $D_\lambda=\emptyset$ if $\lambda$ is small enough. 
	Thus $w_\lambda\ge 0$ in $\Sigma_\lambda$ if $\lambda$ is small. Define
	$$
	\lambda^*=\sup\{ \lambda>0: w_\lambda\ge 0 \hbox{ in } \Sigma_\lambda\}.
	$$
	If we assume that $\lambda^*<+\infty$, then there exist sequences $\lambda_n \downarrow \lambda^*$ and 
	$x_n \in [0,\lambda_n]$ such that $w_{\lambda_n}(x_n)<0$. The maximum principle in narrow 
	domains also implies that the points $x_n$ can be chosen indeed in some interval $[\delta,\lambda^*-\delta]$. 
	Thus we may assume $x_n\to x_0\in [\delta,\lambda^*-\delta]$. 
	
	Passing to the limit we see that $w_{\lambda^*}\ge 0$ in $[0,\lambda^*]$, with $w_{\lambda^*}(x_0)=0$. 
	The strong maximum principle then gives $w_{\lambda^*}\equiv 0$ in $[0,\lambda^*]$, that is, $u$ is 
	symmetric with respect to the point $x=\lambda^*$. However, this contradicts Theorem 8 in 
	\cite{MR3624935}, whose proof can be seen to be valid when $N=1$ as well. 
	
	The contradiction shows that $\lambda^*=+\infty$, that is, $w_{\lambda}\ge 0$ in $[0,\lambda]$ 
	for every $\lambda>0$. Thus $u$ is nondecreasing. Finally, arguing as in Step 3 in the proof of 
	Theorem 1 in \cite{MR3624935}, we see that $u'>0$ in $(0,+\infty)$, as wanted.
	\end{proof}
	
	\bigskip
	
	Next, we will give the first step in obtaining our energy. The following result is somehow 
		related to the ones obtained in \cite{MR3211861} regarding Pohozaev's identity for the fractional laplacian. 
	
	\begin{lemma}\label{lema-integr}
	Let $u\in C(\R) \cap L^\infty(\R) \cap C^1(0,+\infty)$ be such that $u' \in L^1(b_0,+\infty)$ for some $b_0>0$ and 
	$\| u\|_{C^{2s+\beta}[b,+\infty)}$ is finite for every $b>0$ and some $\beta\in (0,1)$. Then 
	\begin{equation}\label{ener-1}
	\begin{array}{rl}
	\ds \int_a^{+\infty} u'(x) (-\Delta)^s u(x) dx \hspace{-2mm} & = 
	- \ds \frac{c(1,s)}{2} \left( \int_{-\infty}^{+\infty} \frac{(u(a)-u(y))^2}{|a-y|^{1+2s}} dy \right.\\[1.25pc]
	& \ds \left. -(1+2s)\int_a^{+\infty} \hspace{-2mm} \int_{-\infty}^a \frac{(u(x)-u(y))^2}{|x-y|^{2+2s}} dy dx\right)
	\end{array}
	\end{equation}
	for every $a>0$. The first integral above is absolutely convergent. In particular, if $u$ is a positive 
	bounded solution of \eqref{eq-unidim} with a locally Lipschitz $f$ then 
	\begin{equation}\label{ener-2}
	\begin{array}{rl}
	F(u(a)) \hspace{-3mm} & - \ds \frac{c(1,s)}{2} \left( \int_{-\infty}^{+\infty} \frac{(u(a)-u(y))^2}{|a-y|^{1+2s}} dy 
	-(1+2s)\hspace{-1mm}\int_a^{+\infty}\hspace{-2mm} \int_{-\infty}^a \frac{(u(x)-u(y))^2}{|x-y|^{2+2s}} dy dx\right)\\[1.25pc]
	& =F(\rho),
	\end{array}
	\end{equation}
	for every $a>0$, where $\rho=\lim_{x\to +\infty} u(x)$ and $F$ is a primitive of $f$.
	\end{lemma}
	
	\begin{proof}
	Fix $a>0$ and choose $\delta$ and $M$ with the restrictions $0<\delta<a$ and $M>a+\delta$. We first 
	consider the integral
	\begin{equation}\label{Idelta}
	I_{\delta,M}= \int_a^M u'(x) \int^M \limits_{-\text{\em\scriptsize M} \atop {|y-x|\ge \delta}} \frac{u(x)-u(y)}{|x-y|^{1+2s}} dy dx 
	= \iint_{A_{\delta,M}}u'(x) \frac{u(x)-u(y)}{|x-y|^{1+2s}} dy dx,
	\end{equation}
	where $A_{\delta,M}=([a,M]\times [-M,M]) \cap \{(x,y)\in \R^2:\ |y-x|\ge \delta\}$ (see Figure 1).
	It is not hard to see that
	\begin{align*}
	I_{\delta,M} & = \frac{1}{2} \iint_{A_{\delta,M}} \frac{\left((u(x)-u(y))^2\right)_x }{|x-y|^{1+2s}}dy dx\\
	& =\frac{1}{2} \iint_{A_{\delta,M}} \hspace{-2mm} \left(\frac{(u(x)-u(y))^2}{|x-y|^{1+2s}}\right)_x 
	\hspace{-1mm} dy dx + \frac{1+2s}{2} \hspace{-1mm} 
	\iint_{A_{\delta,M}} \hspace{-3mm} \frac{(x-y) (u(x)-u(y))^2}{|x-y|^{3+2s}} dy dx.
	\end{align*}
	
	\psfrag{A}{$A_{\delta,M}^1$}
	\psfrag{B}{$A_{\delta,M}^2$}
	\psfrag{C}{$A_{\delta,M}^3$}
	\psfrag{y}{$y=x$}
	\psfrag{a}{$a$}
	\psfrag{M}{$M$}
	\psfrag{-M}{$-M$}
	
	\smallskip
	
	\begin{center}
	\includegraphics[width=5.5cm]{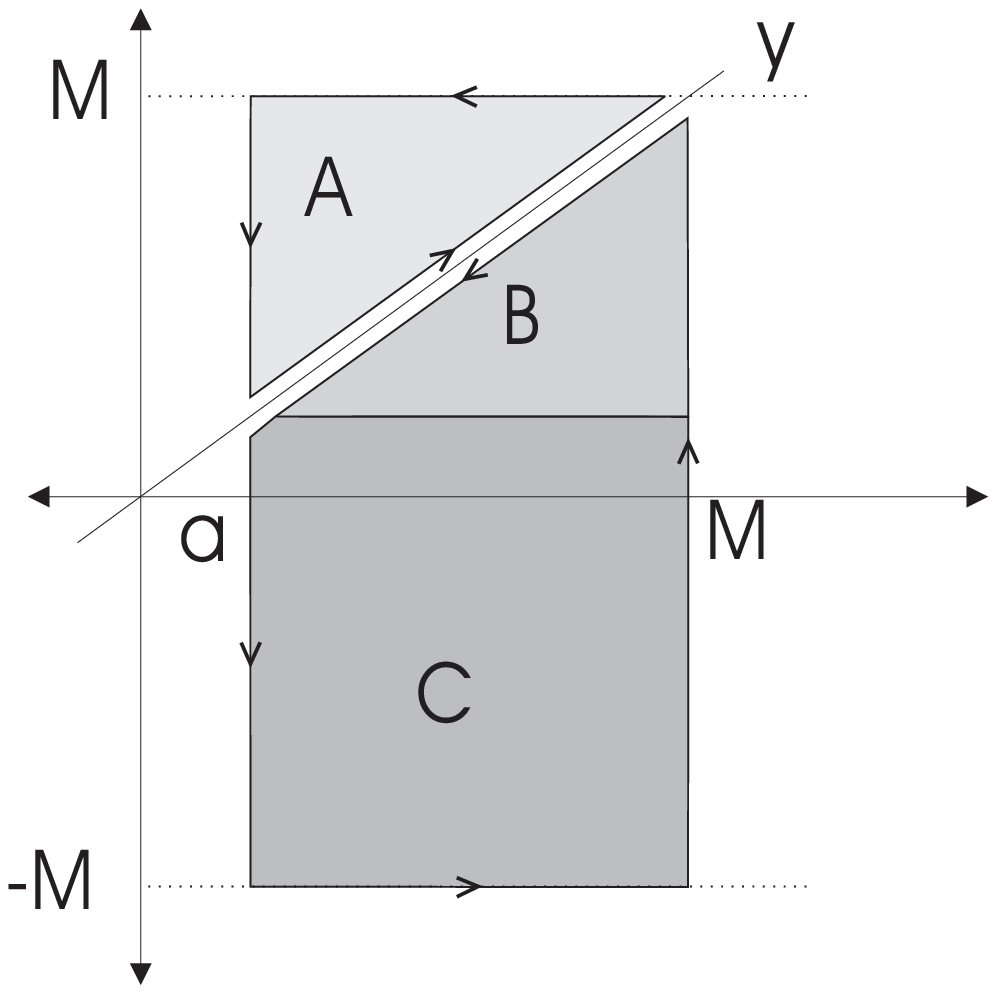}
	\end{center}
	
	\centerline{{\sc Figure 1.} The region $A_{\delta,M}$ and its subregions.}
	
	\bigskip

	We now split $A_{\delta,M}=A_{\delta,M}^1 \cup A_{\delta,M}^2 \cup A_{\delta,M}^3$, where
	$$
	\begin{array}{l}
	A_{\delta,M}^1= \{(x,y) \in A_{\delta,M}: y \ge x+\delta\}\\[0.25pc]
	A_{\delta,M}^2= \{(x,y) \in A_{\delta,M}: a\le y \le x-\delta\}\\[0.25pc]
	A_{\delta,M}^3= \{(x,y) \in A_{\delta,M}: y\le a\}.
	\end{array}
	$$
	Since the region $A_{\delta,M}^1$ is the reflection of $A_{\delta,M}^2$ with 
	respect to the line $y=x$ and the integrand in the last integral above is antisymmetric, we 
	immediately deduce that 	
	\begin{align*}
	I_{\delta,M} & = \frac{1}{2}\iint_{A_{\delta,M}} \left(\frac{(u(x)-u(y))^2}{|x-y|^{1+2s}}\right)_x dy dx + 
	\frac{1+2s}{2}  \iint_{A_{\delta,M}^3} \frac{(u(x)-u(y))^2}{(x-y)^{2+2s}} dy dx\\
	& =\frac{1}{2}\oint_{\partial A_{\delta,M}} \frac{(u(x)-u(y))^2}{|x-y|^{1+2s}}dy+ \frac{1+2s}{2} \iint_{A_{\delta,M}^3} 
	\frac{(u(x)-u(y))^2}{(x-y)^{2+2s}} dy dx.
	\end{align*}
	We have made use of Green's formula, hence the line integral is to be taken in the positive sense. Parameterizing 
	the line integral we deduce
	\begin{align}\label{eq-energia-1}
	I_{\delta,M} & = -\frac{1}{2}\int^M \limits_{-\text{\em\scriptsize M}  \atop |y-a|\ge \delta} \frac{(u(a)-u(y))^2}{|a-y|^{1+2s}} dy 
	+\frac{1}{2} \int_{-M}^{M-\delta} \frac{(u(M)-u(y))^2}{|M-y|^{1+2s}} dy \nonumber \\
	& + \frac{1}{2}\int_a^{M-\delta} \frac{(u(x)-u(x+\delta))^2}{\delta^{1+2s}} dx 
	-\frac{1}{2}\int_a^{M} \frac{(u(x)-u(x-\delta))^2}{\delta^{1+2s}} dx \nonumber \\[0.25pc]
	& + \frac{1+2s}{2} \iint_{A_{\delta,M}^3} 
	\frac{(u(x)-u(y))^2}{(x-y)^{2+2s}} dy dx\nonumber \\
	& =  -\frac{1}{2}\int^M \limits_{-\text{\em\scriptsize M} \atop |y-a|\ge \delta} \frac{(u(a)-u(y))^2}{|a-y|^{1+2s}} dy 
	+\frac{1}{2} \int_{-M}^{M-\delta} \frac{(u(M)-u(y))^2}{|M-y|^{1+2s}} dy \nonumber \\
	& - \frac{1}{2}\int_{a-\delta}^a  \frac{(u(x+\delta)-u(x))^2}{\delta^{1+2s}} dx 
		 + \frac{1+2s}{2} \iint_{A_{\delta,M}^3} \frac{(u(x)-u(y))^2}{(x-y)^{2+2s}} dy dx. \nonumber\\
	& =:I_1+I_2+I_3+I_4.
	\end{align}
	The next step is to let $M\to +\infty$ in \eqref{eq-energia-1}. Since $u$ is bounded 
	we may easily pass to the limit in $I_{\delta, M}$, given in \eqref{Idelta}, $I_1$ and $I_4$ by simply using dominated convergence. As for $I_2$, we claim that it goes to zero as $M\to +\infty$. 
	
	To prove this claim, choose $M_0>a$ and let $M>M_0+\delta$. Then we can write, with the use of 
	the fundamental theorem of calculus and Fubini's theorem:
	\begin{align*}
	\int_{M_0}^{M-\delta} \frac{(u(M)-u(y))^2}{(M-y)^{1+2s}} dy & \le 
	2 \| u \|_{L^\infty(\R_+)} \int_{M_0}^{M-\delta} \frac{|u(M)-u(y)|}{(M-y)^{1+2s}} dy \\
	& \le 2 \| u \|_{L^\infty(\R_+)} \int_{M_0}^{M-\delta} \int_y ^M \frac{|u'(\xi)|}{(M-y)^{1+2s}} d\xi dy\\
	& \le 2 \| u \|_{L^\infty(\R_+)} \int_{M_0}^{M-\delta} \int_{M_0} ^M \frac{|u'(\xi)|}{(M-y)^{1+2s}} d\xi dy\\
	& = 2 \| u \|_{L^\infty(\R_+)} \int_{M_0} ^M \int_{M_0}^{M-\delta}\frac{|u'(\xi)|}{(M-y)^{1+2s}} dy d\xi \\
	& \le \frac{\| u \|_{L^\infty(\R_+)}}{s\delta^{2s}} \int_{M_0} ^{+\infty} |u'(\xi)| d\xi.
	\end{align*}
	On the other hand, 
	\begin{align*}
	\int_{-M}^{M_0} \frac{(u(M)-u(y))^2}{(M-y)^{1+2s}} dy & \le 4 \|u\|_{L^\infty(\R_+)}^2 \int_{-M}^{M_0} 
	\frac{dy}{(M-y)^{1+2s}} dy\\
	& = \frac{2}{s} \|u\|_{L^\infty(\R_+)}^2 (M-M_0)^{-2s}.
	\end{align*}
	Hence
	$$
	I_2 \le 	\frac{\| u \|_{L^\infty(\R_+)}}{s\delta^{2s}} \int_{M_0} ^{+\infty} |u'(\xi)| d\xi+
	\frac{2}{s} \|u\|_{L^\infty(\R_+)}^2 (M-M_0)^{-2s}.
	$$
	Letting $M\to +\infty$ and then $M_0\to +\infty$, we see that the integral goes to zero, as required. 
	Passing to the limit in \eqref{eq-energia-1} and using dominated convergence we see that
	\begin{align}\label{eq-energia-2}
	\int_a^{+\infty} u'(x) \int_{|y-x|\ge \delta} \frac{u(x)-u(y)}{|x-y|^{1+2s}} dy dx 
	& = -\frac{1}{2}\int_{|y-a|\ge \delta} \frac{(u(a)-u(y))^2}{|a-y|^{1+2s}} dy \nonumber \\
	& - \frac{1}{2}\int_{a-\delta}^a  \frac{(u(x+\delta)-u(x))^2}{\delta^{1+2s}} dx \\
	& + \frac{1+2s}{2} \iint_{A_\delta} \frac{(u(x)-u(y))^2}{(x-y)^{2+2s}} dy dx, \nonumber 
	\end{align}
	where $A_\delta=([a,+\infty)\times (-\infty,a]) \cap \{(x,y)\in \R^2: y \le x-\delta\}$. 
	
	The final step will be to pass to the limit as $\delta \to 0$ in \eqref{eq-energia-2}. Observe 
	that, since $u\in C^1 (0,+\infty)$, we have for $y$ close to $a$
	$$
	\frac{(u(a)-u(y))^2}{|a-y|^{1+2s}} \le C |a-y|^{1-2s}\in L^1_{\rm loc}(\R),
	$$
	so the passing to the limit is justified in the first integral in the right-hand side of 
	\eqref{eq-energia-2} by dominated convergence. As for the second integral, we see that, 
	also because of the regularity of $u$:
	$$
	\int_{a-\delta}^a \frac{(u(x+\delta)-u(x))^2}{\delta^{1+2s}} dx \le C\delta^{2-2s}\to 0
	$$
	as $\delta \to 0^+$. As for the double integral, it also follows that 
	$$
	\frac{(u(x)-u(y))^2}{|x-y|^{2+2s}} \le C |x-y|^{-2s}\in L^1_{\rm loc}(\R^2),
	$$
	for $x$ and $y$ close to $a$. Therefore, we are allowed to pass to the limit in the 
	right-hand side of \eqref{eq-energia-2}. 
	
	However, the left-hand side of \eqref{eq-energia-2} has to be treated with a little more care, 
	although in a standard way. By dominated convergence, it suffices to show that
	\begin{equation}\label{eq-energia-3}
	\left| \int_{|y-x|\ge \delta} \frac{u(x)-u(y)}{|x-y|^{1+2s}} dy \right| 
	\le C
	\end{equation}
	for some positive constant $C$ and every $x>a$. First, notice that for $\delta<\frac{a}{2}$:
	\begin{align*}
	\int_{|y-x|\ge \delta} \frac{u(x)-u(y)}{|x-y|^{1+2s}} dy & = 
	\frac{1}{2} \int_{|z|\ge \delta} \frac{2u(x)-u(x+z)-u(x-z)}{|z|^{1+2s}} dz\\
	& = \frac{1}{2} \left(\int_{\delta \le |z| \le \frac{a}{2}} + \int_{|z|>\frac{a}{2}} \right)
	\frac{2u(x)-u(x+z)-u(x-z)}{|z|^{1+2s}} dz.
	\end{align*}
	The absolute value of the second of these integrals can be estimated by 
	$$
	2\| u \|_{L^\infty(\R_+)} \int_{|z|>\frac{a}{2}} \frac{dz}{|z|^{1+2s}}.
	$$
	To estimate the first integral, we recall our hypothesis 
	that $\|u\|_{C^{2s+\beta}[b,+\infty)}$ is finite for some $\beta\in (0,1)$ and 
	every $b>0$. Since $x>a$, it follows that
	$$
	\int_{\delta \le |z| \le \frac{a}{2}} \left| \frac{2u(x)-u(x+z)-u(x-z)}{|z|^{1+2s}} \right| dz 
	\le C \|u\|_{C^{2s+\beta}[\frac{a}{2},+\infty)} 	\int_{|z| \le \frac{a}{2}} |z|^{\beta-1} dz,
	$$
	for some (explicit) $C>0$. Thus \eqref{eq-energia-3} follows.
	
	To summarize, we may pass to the limit as $\delta\to 0^+$ in \eqref{eq-energia-2}, and 
	\eqref{ener-1} follows just multiplying by $c(1,s)$. 
	
	To conclude the proof of the lemma, let $u$ be a positive, bounded solution of \eqref{eq-unidim}. 
	By Lemma \ref{lema-monotonia}, $u'>0$ so that $u'\in L^1(1,+\infty)$, say. On the other hand, by 
	standard regularity we obtain that $u\in C^1(0,+\infty)$ and that the $C^{2s+\beta}$ norm of 
	$u$ in any interval of the form $[b,+\infty)$ is bounded for every $\beta \in (0,1)$. 	
	Thus the first part of the proof applies and we obtain \eqref{ener-2} by just noticing that 
	$$
	\int_a^{+\infty} u'(x) (-\Delta)^s u(x) dx =F (\rho)-F(u(a)),
	$$
	where $\rho=\lim_{x\to +\infty} u(x)$.
	\end{proof}
	
	\medskip
	
	Our next result is obtained by letting $a\to 0^+$ in \eqref{ener-2}. We use ideas in 
		Theorem 7.5 of \cite{MR2214908}.

	\begin{lemma}\label{lema-constante-1}
	Let $u$ be a bounded positive solution of \eqref{eq-unidim}. Then 
	\begin{equation}\label{eq-prel}
	F(\rho) = \mathcal{K}(s) \ell_0^2,
	\end{equation}
	where $\rho=\lim_{x\to +\infty} u(x)$, $\ell_0$ is given in \eqref{def-elecero} and 
	\begin{equation}\label{Csb}
	\begin{array}{rl}
	\mathcal{K}(s) \hspace{-2mm} & = \ds \frac{c(1,s)}{2} \left( -\frac{1}{2s} -\int_{-1}^1 \frac{((t+1)^s -1)^2}{|t|^{1+2s}} dt
	+ \int_1^{+\infty} \frac{t^{2s} - ((t+1)^s-1)^2}{t^{1+2s}} dt\right.\\[1pc]
	& \left. \ds \qquad \qquad +(1+2s)\int_1^{+\infty} \hspace{-2mm} \int_0^1 \frac{(t^s-\tau^s)^2}{(t-\tau)^{2+2s}} 
	d\tau dt \right).
	\end{array}
	\end{equation}
	\end{lemma}
	
	\begin{proof}
	All the integrals in \eqref{Csb} can be seen to be convergent (but see the proof of 
	Lemma \ref{calculo-constante} below). 
	
	We first remark that, by boundary regularity, the function $\frac{u(x)}{x^s}$ is 
	in $C^1[0,+\infty)$ (cf. Theorem 7.4, part (iii) in \cite{MR3447732}). Thus in particular the 
	value $\ell_0$ given in \eqref{def-elecero} is well defined. 
	
	Let $a>0$. Since $u=0$ in $(-\infty,0)$, we can write
	\begin{align*}
	\int_{-\infty}^{+\infty} \frac{(u(a)-u(y))^2}{|a-y|^{1+2s}} dy & = 
	\int_{-\infty}^0 \frac{u(a)^2}{|a-y|^{1+2s}} dy + 
	\int_0^{+\infty} \frac{(u(a)-u(y))^2}{|a-y|^{1+2s}} dy\\
	& = \frac{1}{2s} \frac{u(a)^2}{a^{2s}} + 
	\int_{-a}^{+\infty} \frac{(u(a)-u(z+a))^2}{|z|^{1+2s}} dz.
	\end{align*}
	Similarly
	$$
	\int_a^{+\infty} \hspace{-2mm} \int_{-\infty}^a \hspace{-1mm} \frac{(u(x)-u(y))^2}{|x-y|^{2+2s}} dy dx =
	\frac{1}{1+2s} \int_a^{+\infty} \hspace{-1mm} \frac{u(x)^2}{x^{1+2s}} dx + 
	\int_a^{+\infty} \hspace{-2mm} \int_0^a \hspace{-1mm} \frac{(u(x)-u(y))^2}{|x-y|^{2+2s}} dy dx.
	$$
	Thus by \eqref{ener-2} we see that 
	\begin{equation}\label{eq-pass}
	\begin{array}{rl}
	F(\rho) \hspace{-2mm} & = \ds  F(u(a)) -\frac{c(1,s)}{4s} \frac{u(a)^2}{a^{2s}}  -\frac{c(1,s)}{2}\int_{-a}^{+\infty}
	\frac{(u(a)-u(z+a))^2}{|z|^{1+2s}} dz\\
	& + \ds \frac{c(1,s)}{2} \int_a^{+\infty} \frac{u(x)^2}{x^{1+2s}} dx + \frac{c(1,s)(1+2s)}{2} 
	\int_a^{+\infty} \hspace{-2mm} \int_0^a \frac{(u(x)-u(y))^2}{|x-y|^{2+2s}} dy dx.
	\end{array}
	\end{equation}
	Our intention is to pass to the limit in this equality as $a\to 0^+$. For this sake, it is clear that 
	only the integrals need to be taken into account. 
	
	We first claim that 
	\begin{equation}\label{claim-1-ks}
	\lim_{a\to 0} \int_a^{+\infty} \hspace{-2mm} \int_0^a \frac{(u(x)-u(y))^2}{|x-y|^{2+2s}} dy dx 
	= \ell_0^2 \int_1^{+\infty} \hspace{-2mm} \int_0^1 \frac{(t^s-\tau^s)^2}{(t-\tau)^{2+2s}} d\tau dt.
	\end{equation}
	To prove \eqref{claim-1-ks}, fix $\eta>0$ and take $a<\frac{\eta}{2}$. Then for $x>\eta$ and 
	$0<y<a$ we have $x-y\ge \frac{x}{2}$. Therefore
	\begin{align}\label{claim-1-eq-1}
	\int_\eta^{+\infty} \hspace{-2mm} \int_0^a \frac{(u(x)-u(y))^2}{|x-y|^{2+2s}} dy dx  & 
	\le 4 \| u\|_{L^\infty(\R)}^2 \int_\eta^{+\infty} \hspace{-2mm} \int_0^a \frac{dy}{|x-y|^{2+2s}} dx \nonumber \\
	& \le 2^{4+2s}  \| u\|_{L^\infty(\R)}^2 a  \int_\eta^{+\infty} x^{-2-2s} dx\\
	& = \frac{2^{4+2s}  \| u\|_{L^\infty(\R)}^2}{1+2s} \eta^{-1-2s} a. \nonumber 
	\end{align}
	To analyze the same integral when $x$ varies in the interval $[a,\eta]$, observe that 
	the regularity of $u(x)/x^s$ implies
	$$
	\lim_{x\to 0} \frac{u'(x)}{x^{s-1}}= s\ell_0.
	$$
	Therefore, if we fix $\varepsilon>0$, for small enough $\eta$ we can guarantee that 
	$u'(x)\le s(\ell_0+\varepsilon)x^{s-1}$ if $x<\eta$. Hence for $y<a<x<\eta$ we have
	$$
	0< u(x) -u(y) = \int_y^x u'(\xi) d\xi \le (\ell_0+\varepsilon) (x^s-y^s),
	$$
	so that 
	$$
	\int_a^\eta \hspace{-2mm} \int_0^a \frac{(u(x)-u(y))^2}{|x-y|^{2+2s}} dy dx 
	\le (\ell_0+\varepsilon)^2 
	\int_a^\eta \hspace{-2mm} \int_0^a \frac{(x^s-y^s)^2}{|x-y|^{2+2s}} dy dx.
	$$
	In the last integral, we change variables by $x=at$, $y=a \tau$ and recall 
	\eqref{claim-1-eq-1} to obtain, for some $C>0$,
	$$
	\int_a^{+\infty} \hspace{-2mm} \int_0^a \frac{(u(x)-u(y))^2}{|x-y|^{2+2s}} dy dx \le C \eta^{-1-2s} a + 
	 (\ell_0+\varepsilon)^2 \int_1^{\frac{\eta}{a}} \hspace{-2mm} \int_0^1 \frac{(t^s-\tau^s)^2}{(t-\tau)^{2+2s}} 
	d\tau dt.
	$$
	Letting $a\to 0^+$ and then $\varepsilon \to 0^+$ we have 
	$$
	\limsup_{a\to 0} \int_a^{+\infty} \hspace{-2mm} \int_0^a \frac{(u(x)-u(y))^2}{|x-y|^{2+2s}} dy dx 
	\le \ell_0^2 \int_1^{+\infty} \hspace{-2mm} \int_0^1 \frac{(t^s-\tau^s)^2}{(t-\tau)^{2+2s}} d\tau dt.
	$$
	The opposite inequality for the inferior limit is shown similarly, and this establishes 
	\eqref{claim-1-ks}. 
	
	We finally deal with the remaining two integrals in \eqref{eq-pass}. We write
	\begin{align*}
	-\int_{-a}^{+\infty} \frac{(u(z+a)-u(a))^2}{|z|^{1+2s}} dz
	+  \int_a^{+\infty} \frac{u(z)^2}{z^{1+2s}} dz & = 
	-\int_{-a}^a \frac{(u(z+a)-u(a))^2}{|z|^{1+2s}} dz \\
	& +	\int_a ^{+\infty} \frac{u(z)^2-(u(z+a)-u(a))^2}{|z|^{1+2s}} dz.
	\end{align*}
	Reasoning exactly as with \eqref{claim-1-ks} it can be shown that 
	$$
	\lim_{a\to 0} \int_{-a}^a \frac{(u(z+a)-u(a))^2}{|z|^{1+2s}} dz =
	\ell_0^2 \int_{-1}^1 \frac{((t+1)^s -1)^2}{|t|^{1+2s}} dt.
	$$
	On the other hand, using the $C^1$ regularity of $u(x)/x^s$ up to $x=0$ we can ensure 
	that 
	\begin{equation}\label{expansion-boundary}
	u(x)= \ell_0 x^s +O(x^{s+1}), \quad \hbox{as }x\to 0,
	\end{equation}
	where $O(x)$ is as usual a function which verifies $|O(x)| \le C x$ for small $x$ 
	and some $C>0$. It follows from \eqref{expansion-boundary} that for small $\eta>0$, if $a<z<\eta$,
	$$
	(u(z+a)-u(a))^2 = \ell_0^2 ((z+a)^s-a^s)^2 + O(z^{s+1}).
	$$
	Thus if $\eta>0$ is small enough and $a<\eta$:
	\begin{align*}
	\int_a ^\eta \hspace{-1mm} \frac{u(z)^2-(u(z+a)-u(a))^2}{z^{1+2s}} dz & = 
	\ell_0^2 \int_a^\eta \frac{z^{2s} - ((z+a)^s-a^s)^2 + O(z^{s+1})}{z^{1+2s}} dz\\
	& = \ell_0^2 \int_a^\eta\frac{z^{2s} - ((z+a)^s-a^s)^2}{z^{1+2s}} dz + 
	\int_a^\eta O(z^{-s})dz \\
	& = \ell_0^2 \int_1^{\frac{\eta}{a}} \frac{t^{2s} - ((t+1)^s-1)^2}{t^{1+2s}} dt + 
	O(\eta^{1-s}).
	\end{align*}
	Moreover, by dominated convergence:
	$$
	\lim_{a\to 0^+} \int_\eta^{+\infty} \frac{u(z)^2-(u(z+a)-u(a))^2}{z^{1+2s}} dz =0.
	$$
	Hence we deduce
	$$
	\lim_{a\to 0^+} \int_a^{+\infty} \frac{u(z)^2-(u(z+a)-u(a))^2}{z^{1+2s}} dz = 
	\ell_0^2 \int_1^{+\infty} \frac{t^{2s} - ((t+1)^s-1)^2}{t^{1+2s}} dt.
	$$
	Finally, we can pass to the limit in \eqref{eq-pass} to conclude the proof of 
	the lemma.	
	\end{proof}
	
	\bigskip
	
	Our last step is to obtain an alternative expression for the constant in \eqref{eq-prel}. To do it, 
	we take advantage of some of the results in \cite{MR3211861}, complemented 
	with an additional analysis of the properties of $\mathcal{K}(s)$.

\begin{lemma}\label{calculo-constante}
For $s\in (0,1)$ we have 
	\begin{equation}\label{igualdad-ctes-1}
		\mathcal{K}(s) = \frac{\Gamma(1+s)^2}{2},
	\end{equation}
where $\mathcal{K}(s)$ is given in \eqref{Csb}.
\end{lemma}

\begin{proof}	
Let us begin by proving \eqref{igualdad-ctes-1} for $s>\frac{1}{2}$. This will follow by 
establishing \eqref{eq-prel} for a particular problem in two different ways. 
For $\lambda>0$ to be chosen later, consider the problem 
\begin{equation}\label{eq-max-calculo}
	\begin{cases}
		(-\Delta)^s u=\lambda( 1-u) &\text{ in } \R_+,\\
		\ \ u=0 &\text{ in } \R\setminus \R_+.\\
	\end{cases}
\end{equation}
	By Theorem \ref{theorem:MaximalSolR}, problem \eqref{eq-max-calculo} admits a maximal solution relative 
	to $\overline{u}=1$, which will be denoted by $u$. The function $u$ is strictly increasing and verifies 
	$\lim_{x\to +\infty} u(x)=1$. 
	
	We claim that $f(u):=\lambda(1-u)\in L^1(0,+\infty)$. To prove this we will construct a suitable subsolution of 
	\eqref{eq-max-calculo}. Choose a nondecreasing function $v\in C^\infty(\R)$ verifying
	
$$
	v(x)=\begin{cases}
		0 &\text{ in } (-\infty,0],\\
		1-x^{-2s} &\text{ if } x\ge 2.\\
	\end{cases}
$$
Then, for $x\ge 4$:
		\begin{align}
		(-\Delta)^s v(x)&=c(1,s) \left(
		\int_{-\infty}^{0}\frac{1- x^{-2s}}{|x-y|^{1+2s}}dy
		+\int_{0}^{2}\frac{(1- x^{-2s})-v(y)}{|x-y|^{1+2s}}\right.
		\, dy\nonumber\\
		&\qquad- \left. \int_{2}^{+\infty}
		\frac{x^{-2s}-y^{-2s}}{|x-y|^{1+2s}}\, dy\right) \nonumber\\
	&= c(1,s) x^{-2s} \hspace{-1mm}\left(\hspace{-1mm}(1- x^{-2s}) \hspace{-1mm}\int_{-\infty}^{0}\frac{d\tau}{|1-\tau|^{1+2s}}
	+ \hspace{-1mm} \int_{0}^{2/x} \hspace{-1mm} \frac{(1- x^{-2s})-v(\tau\, x)}{|1-\tau|^{1+2s}}\, d\tau\right.\nonumber\\
	&\qquad- \, \,  \left. 	x^{-2s} \int_{2/x}^{+\infty}
\frac{\tau^{-2s}-1}{|1-\tau|^{1+2s}}\, d\tau\right)\nonumber\\
	& \le c(1,s) x^{-2s}\left(\int_{-\infty}^{1/2}\frac{d\tau}{|1-\tau|^{1+2s}} - 
	x^{-2s} \int_{1/2}^{+\infty} \frac{\tau^{-2s}-1}{|1-\tau|^{1+2s}}\, d\tau\right)
		,	\label{cte_explicita1}
	\end{align}
	where we have made the change of variables $\tau =y/x$ in the first three integrals above. Observe 
	that the last integral converges, since it is to be understood in the principal value sense, 
	as always. It follows from \eqref{cte_explicita1} that for some $C>0$
	$$
	(-\Delta)^s v(x)\leq C x^{-2s}, \qquad \hbox{for } x\ge 4.
	$$
	Since $v$ is a smooth function, the same inequality holds for $x\ge 2$, by enlarging the 
	constant if necessary. Therefore, if $\lambda$ is large enough we see that 
	$$
	(-\Delta)^s v (x) \le \lambda (1-v(x)), \qquad \hbox{for } x\ge 2.
	$$
	On the other hand, the monotonicity of $v$ implies that $v$ is bounded away from 1 in 
	the interval $[0,2]$, hence the same inequality can be achieved there by taking a larger 
	value of $\lambda$. 
		
	Thus we have shown that $v$ is a subsolution of \eqref{eq-max-calculo} if $\lambda$ is large 
	enough. It follows by the maximality of $u$ that $v\le u$ in $\R$, therefore, if $x\ge 2$:
	$$
	1-u(x) \le 1-v(x) = x^{-2s} \in L^1 (2,+\infty),
	$$
	since $s>\frac{1}{2}$, which completes the proof of the claim.
	
	We now apply Lemma \ref{lema-constante-1} to problem \eqref{eq-max-calculo} to 
	obtain
	\begin{equation}\label{cte-1}
	F(1)=\frac{\lambda}{2}= \mathcal{K}(s) \ell_0^2,
	\end{equation}
	where $\ell_0=\lim_{x\to 0} u(x)/x^{s}$. 
	
	On the other hand, we now make the crucial observation that some of the results in \cite{MR3211861} 
	can be applied to solutions $u$ of problems posed in unbounded domains $\Omega$ 
	as long as $f(u)\in L^1(\Omega)$, which is precisely the situation in \eqref{eq-max-calculo}. More 
	precisely, see the proof of Proposition 1.6 and (2.7) there. 
	In particular by Theorem 1.9 in \cite{MR3211861} we see that 
	\begin{equation}\label{cte-2}
	\frac{\lambda}{2}= \frac{\Gamma(1+s)^2}{2} \ell_0^2.
	\end{equation}
	Combining \eqref{cte-1} and \eqref{cte-2} we see that \eqref{igualdad-ctes-1} holds for $s>\frac{1}{2}$.
	
	Unfortunately, this procedure does not seem to be generalized to cover the whole range 
	$s\in (0,1)$. Indeed, we expect the maximal solution $u$ of \eqref{eq-max-calculo} to 
	behave exactly like $1-x^{-2s}$ as $x\to +\infty$, so that $f(u)\not\in L^1(0,+\infty)$ if 
	$s\le \frac{1}{2}$. 
	
	\medskip
	
	Therefore we will prove \eqref{igualdad-ctes-1} by showing that $\mathcal{K}(s)$ 
	can be seen as an analytic function of the complex variable $s$ in the strip $0<\text{Re}(s)<1$. 
	Since it coincides with $\Gamma(s+1)^2/2$ in the real segment $(\frac{1}{2},1)$, the well-known identity principle 
	will imply that both functions coincide throughout the strip, therefore in the segment $(0,1)$.
	
	First of all, we write the function $\mathcal{K}(s)$ 
	as follows
	\begin{equation}\label{eq:ks}
		\mathcal{K}(s) = \frac{c(1,s)}{2} \left( -\frac{1}{2s} -F_1(s) 
	+ F_2(s)  +(1+2s)F_3(s)\right), 
	\end{equation}
	where
	\begin{equation}\label{def-efes}
	\begin{array}{l}
	\ds F_1(s):= \int_{-1}^1 \frac{((t+1)^s -1)^2}{|t|^{1+2s}} dt\\[.75pc]
	\ds F_2(s):= \int_1^{+\infty} \frac{t^{2s} - ((t+1)^s-1)^2}{t^{1+2s}} dt\\[0.75pc]
	\ds F_3(s):= \int_1^{+\infty} \hspace{-2mm} \int_0^1 \frac{(t^s-\tau^s)^2}{(t-\tau)^{2+2s}} d\tau dt.
	\end{array}
	\end{equation}
	Therefore, it suffices to verify that $F_1$, $F_2$ and $F_3$ are analytic in the strip 
	$0<\text{Re}(s)<1$. We will achieve this by showing that each of the integrals in  
  \eqref{def-efes} converges absolutely and uniformly in rectangles of the form $U_{\sigma_1,\sigma_2,K}=\{s\in \mathbb{C}: \ 
	\sigma_1 \le \text{Re}(s)\le \sigma_2,\ -K \le \text{Im}(s) \le K\}$, where $0<\sigma_1<\sigma_2<1$ and $K>0$.
	
	We use the notation $s=\sigma + i\omega$, where $\sigma_1\le \sigma\le \sigma_2$ and $|\omega|\le K$. 
	It is important to stress that the complex power functions appearing in \eqref{def-efes} have 
	to be understood in the sense
	$$
	x^s= x^\sigma e^{i \omega\log x}, \quad x\in \R_+.
	$$
	Thus in particular $|x^s|=x^\sigma$ for every $x>0$.
	
	We begin with the integral defining $F_1$. It is enough to prove the uniform convergence of 
	the integral in $[-\frac{1}{2},1]$. Observe that, for $s\in U_{\sigma_1,\sigma_2,K}$, $t\in [-\frac{1}{2},1]$:
	\begin{equation}\label{des-anal}
	\begin{array}{rl}
	|(t+1)^s-1|^2 \hspace{-2mm} &  =((t+1)^\sigma - \cos (\omega \log(t+1)))^2 + \sin^2 (\omega\log (t+1))\\[.5pc]
	& \le ((2^{1-\sigma_2}\sigma_2+2K)^2 + K^2) t^2=Ct^2.
	\end{array}
	\end{equation}
	Therefore
	$$
	\int_{-\frac{1}{2}}^1 \left|\frac{((t+1)^s -1)^2}{|t|^{1+2s}}\right| dt 
	= \int_{-\frac{1}{2}}^1 \frac{|(t+1)^s-1|^2}{|t|^{1+2\sigma}} dt \le 
	C \int_{-\frac{1}{2}}^1 |t|^{1-2\sigma_2} dt,
	$$
	which shows the absolute and uniform convergence of the integral, therefore the analyticity of $F_1$. 
	As for $F_2$, we have 
	$$
	|t^{2s}-((t+1)^s-1)^2| \le |t^{s}-(t+1)^s+1||t^s+(t+1)^s-1|\le |s t^{s-1}-1| t^{\sigma} \le C t^{\sigma},
	$$
	thus
	$$
	\int_1^{+\infty} \left| \frac{t^{2s} - ((t+1)^s-1)^2}{t^{1+2s}}\right|  dt
	\le C \int_1^{+\infty} \frac{ dt}{t^{1+\sigma_1}},
	$$
	which shows that $F_2$ is analytic as well.
	
	Finally, we consider the integral defining $F_3$. We split it as follows:
	\begin{align*}
	\int_1^{+\infty} \hspace{-2mm} \int_0^1 \frac{(t^s-\tau^s)^2}{(t-\tau)^{2+2s}} d\tau dt &
	=\int_1^2 \hspace{-2mm} \int_0^\frac{1}{2} \frac{(t^s-\tau^s)^2}{(t-\tau)^{2+2s}} d\tau dt
	+\int_1^2 \hspace{-2mm} \int_\frac{1}{2}^1 \frac{(t^s-\tau^s)^2}{(t-\tau)^{2+2s}} d\tau dt\\
	& +\int_2^{+\infty} \hspace{-2mm} \int_0^1 \frac{(t^s-\tau^s)^2}{(t-\tau)^{2+2s}} d\tau dt
	=:I_1(s)+I_2(s)+I_3(s).
	\end{align*}
	Notice that $I_1$ defines an analytic function since it is a proper 
	integral. Thus we only have to show the uniform convergence of $I_2$ and $I_3$.
	Regarding $I_2$, observe that for $s\in U_{\sigma_1,\sigma_2,K}$, 
	$t\in [1,2]$ and $\tau\in [\frac{1}{2},1]$, we have, reasoning as in \eqref{des-anal}:
	\begin{align*}
	|t^s-\tau^s|^2 & =(t^\sigma-\tau^\sigma \cos(\omega(\log t-\log \tau)))^2 + 
	\sin^2 (\omega(\log t-\log \tau)) \\
	& \le C(t-\tau)^2 \tau^{\sigma-2}\le C(t-\tau)^2,
	\end{align*}
	for some $C>0$. Therefore:
	$$
	\int_1^2 \hspace{-2mm} \int_\frac{1}{2}^1 \left| \frac{(t^s-\tau^s)^2}{(t-\tau)^{2+2s}} \right| d\tau dt
	\le C \int_1^2 \hspace{-2mm} \int_\frac{1}{2}^1 \frac{d\tau}{(t-\tau)^{2\sigma}} dt
	\le C \int_1^2 \hspace{-2mm} \int_\frac{1}{2}^1 \frac{d\tau}{(t-\tau)^{2\sigma_1}} dt.
	$$
	Finally, for the remaining integral $I_3$ we have
	\begin{align*}
	\int_2^{+\infty} \hspace{-2mm} \int_0^1 \left| \frac{(t^s-\tau^s)^2}{(t-\tau)^{2+2s}}\right| d\tau dt & 
	\le \int_2^{+\infty} \hspace{-2mm} \int_0^1 \frac{(t^\sigma+\tau^\sigma)^2}{(t-\tau)^{2+2\sigma}} d\tau dt\\ 
	& \le \int_2^{+\infty} \frac{(t^\sigma+1)^2}{(t-1)^{2+2\sigma}} dt \le 
	36 \int_2^{+\infty} \frac{dt}{t^2},
	\end{align*}
	thereby showing the analyticity of $F_3$. 
	To summarize, we have shown that $F_1$, $F_2$ and $F_3$ define analytic functions in the strip 
	$0<\text{Re}(s)<1$. As we have already remarked, this concludes the proof of \eqref{igualdad-ctes-1}.
\end{proof}	

\bigskip

\begin{proof}[Proof of Theorem \ref{thm-pral-energia}] It is immediate taking into account 
Lemmas \ref{lema-monotonia}, \ref{lema-integr}, \ref{lema-constante-1} and 
	\ref{calculo-constante}.
\end{proof}

\bigskip
	
\subsection{Uniqueness of one-dimensional solutions}	
	We finally come to the principal result of this section which is the uniqueness 
	of positive solutions of \eqref{eq:Problema-unidim}.
		
	\begin{theorem}\label{unicidad_unidim}
	Assume $f$ is locally Lipschitz and $\rho>0$ is such that $f(\rho)=0$ and \eqref{HF1} 
	holds. Then the problem
	\begin{equation}\label{P_unicidad}\tag{$P_1$}
		\begin{cases}
			(-\Delta)^s u = f(u) &\text{ in } \mathbb{R}_+,\\
			\ \ u=0 &\text{ in } \mathbb{R}\setminus \mathbb{R}_+,
			\end{cases}
	\end{equation}
	admits at most a positive solution $u$ verifying
	\begin{equation}\label{norm-inf}
	\|u\|_{L^\infty(\R)}=\rho,
	\end{equation}
	that we will denote by $u_\rho$.
	\end{theorem}
	
\begin{proof}
Let $u$ be a positive solution of \eqref{P_unicidad} verifying \eqref{norm-inf} and denote 
by $\widetilde{u}$ the maximal solution relative to $\rho$ given by Theorem \ref{theorem:MaximalSolR}. 
Then $u\le \widetilde{u}$ in $\R$. Since $(-\Delta)^{s}(\widetilde{u}-u)\geq 
	-L (\widetilde{u}-u)$ in $(0,+\infty)$, where $L$ is the Lipschitz constant of $f$, we deduce 
by the strong maximum principle that either $u\equiv \widetilde{u}$ in $\R$ or $u<\widetilde{u}$ in $(0,+\infty)$.
Let us rule out the second possibility.

Indeed, assume $u<\widetilde{u}$ in $(0,+\infty)$. By Hopf's lemma (see Lemma 1.2 in \cite{MR3533199}) we 
have
	\begin{equation}\label{contradic1}
		\lim_{x\to 0^{+}}\frac{\widetilde{u}(x)-u(x)}{x^s}>0.
	\end{equation}
On the other hand, we may apply Theorem \ref{thm-pral-energia} to have 
		\begin{equation}\label{Fundamental}
		\lim_{x\to 0^{+}} \frac{u(x)}{x^s} = \frac{(2F(\rho))^\frac{1}{2}}{\Gamma(1+s)},
	\end{equation}
	and the same equality holds for $\widetilde{u}$. Hence we deduce  
	$$
		\lim_{x\to 0^{+}}\frac{\widetilde{u}(x)-u(x)}{x^s}=0,
	$$
	which is a contradiction with \eqref{contradic1}. 
	
	Thus we necessarily have $u\equiv \widetilde{u}$ in $\R$, thereby showing that the maximal 
	solution is the only one verifying \eqref{norm-inf}. The proof is concluded.
\end{proof}
	
	\bigskip
	
	%
	%
	\section{Proof of the main results}\label{sec-main}
	\setcounter{equation}{0}
	
	This section is dedicated to prove the main results in the paper. 
	We begin with the proof of the features of problem \eqref{eq:Problema-unidim}.
	
	\bigskip

\begin{proof}[Proof of Theorem \ref{theorem:unidim}]
Let $\rho>0$ such that $f(\rho)=0$ and \eqref{HF1} is satisfied. By Theorem \ref{theorem:MaximalSolR}, 
there exists a positive solution $u_{\rho}$ of \eqref{eq:Problema-unidim} verifying 
$\|u_{\rho}\|_{L^\infty(\R)}=\rho$. Moreover, by 
Theorem \ref{unicidad_unidim}, this is the only solution with this property, and $u_\rho$ is 
strictly increasing and verifies \eqref{eq:elecero}.

Thus to conclude the proof, we need to show that, given any bounded, positive solution $u$ of \eqref{eq:Problema-unidim} 
and setting $\rho=\|u \|_{L^\infty(\R)}$ we necessarily have $f(\rho)=0$ and $f$ verifies \eqref{HF1}.

To show the first assertion, consider the functions
$$
u_n(x) = u(x+n), \quad x\in \R.
$$
It is clear that $u_n$ is a solution of \eqref{eq:Problema-unidim} but posed in the interval $(-n,+\infty)$. 
Since the sequence $\{u_n\}$ is uniformly bounded, we can use interior regularity as in the proof of Theorem
\ref{theorem:MaximalSolR} 
to obtain local $C^\alpha$ bounds, which permit to conclude that, passing to a subsequence, $u_n\to v$ locally uniformly, 
where $v$ is a viscosity solution of 
$$
(-\Delta)^s v= f(v) \qquad \hbox{in } \R.
$$
On the other hand, by Lemma \ref{lema-monotonia}, the function $u$ is monotone. It follows that 
$v\equiv \rho$ in $\R$, and therefore $f(\rho)=0$. 

Finally, let us show that $F(s) <F(\rho)$ for $s\in [0,\rho)$, and the proof of the theorem will be 
concluded. Suppose this is not true. Then there exists a first point $\rho_0 \in (0,\rho)$ such that 
$$
F(\rho_0)=\max_{t\in [0,\rho]}F(t).
$$ 
Thus, in particular, $f(\rho_0)=0$ and \eqref{HF1} holds with $\rho_0$ in place of $\rho$. By Theorem 
\ref{theorem:MaximalSolR} we get a positive solution $v$ of \eqref{eq:Problema-unidim} which is increasing and 
verifies $\| v\|_{L^\infty(\R)}= \rho_0$.

Now we use Theorem \ref{theorem:SubSupV} and Remark \ref{rem-franja} in the Appendix with $v$ as a subsolution 
and $\rho$ as a supersolution and 
obtain a positive solution $w$ verifying $v\le w \le \rho$ in $\R$, which is maximal relative to $\rho$, 
and verifies in particular $\lim_{x\to +\infty} w(x)=\rho$.  
Using the Lipschitz condition on $f$ we have 
$$
(-\Delta)^s (w-v)\ge -L (w-v) \quad \hbox{in }\R_+,
$$
for some $L>0$. By Hopf's Lemma:
	$$
		\lim_{x\to 0^{+}}\frac{w(x)-v(x)}{x^s}>0.
	$$
But, on the other hand, by Theorem \ref{thm-pral-energia}
$$
\lim_{x\to 0^+} \frac{w(x)-v(x)}{x^s} =\frac{\sqrt{2}}{\Gamma(1+s)}(F(\rho)^\frac{1}{2}-F(\rho_0)^\frac{1}{2})
\le 0, 
$$
which is a contradiction. The claim follows.
\end{proof}

\bigskip

\begin{proof}[Proof of Proposition \ref{theorem:exist-half}]
The proof of this result is a consequence of a more general fact: if $v$ is a function defined in 
$\R$ and vanishing in $\R\setminus \R_+$ and we set $u(x)=v(x_N)$ for $x\in \R^N$, then 
$$
(-\Delta)^s u(x)= (-\Delta)^s v(x_N) \quad \text{ in }\R^N,
$$ 
where the first $s-$laplacian is meant to be in $\R^N$ and 
the second one in $\R$. 

To check this fact, observe that by its very definition and Fubini's theorem, we have 
for $x\in \R^N$:
$$
\begin{array}{rl}
(-\Delta)^s u(x) \hspace{-3mm} & = c(N,s)\displaystyle \int_{\R^N} \frac{v (x_N)- v (y_N)}{|x-y|^{N+2s}} dy\\[1pc]
& =c(N,s) \hspace{-1mm} \displaystyle \int_{-\infty}^{+\infty} \hspace{-1mm} (v(x_N)-v(y_N)) \hspace{-1mm} 
\int_{\R^{N-1}} \hspace{-1mm}
\frac{dy'}{(|x'-y'|^2+(x_N-y_N)^2)^\frac{N+2s}{2}} dy_N\\[1pc]
& = \ds c(N,s) \int_{\R^{N-1}} \frac{dz'}{(|z'|^2+1)^\frac{N+2s}{2}}  \int_{-\infty}^{+\infty} \frac{v(x_N)-v(y_N)}{|x_N-y_N|^{1+2s}} 
dy_N,
\end{array}
$$
where we have performed the change of variables $y'=x'+|x_N-y_N|z'$ in the integral taken in $\R^{N-1}$
in the second line above. Thus the proof of the theorem reduces to show that 
\begin{equation}\label{eq-igual-cte}
c(N,s) \int_{\R^{N-1}} \frac{dz'}{(|z'|^2+1)^\frac{N+2s}{2}} =c(1,s).
\end{equation}
With regard to the integral in \eqref{eq-igual-cte}, we have 
$$
\int_{\R^{N-1}} \frac{dz'}{(|z'|^2+1)^\frac{N+2s}{2}} = (N-1) \omega_{N-1} \int_0^{+\infty} 
\frac{r^{N-2}}{(r^2+1)^\frac{N+2s}{2}} dr,
$$
where we denote as usual by $\omega_{N-1}$ the measure of the unit ball in $\R^{N-1}$. 
In the last integral obtained, we perform the change of variables $r=\tan t$ to obtain
$$
\begin{array}{rl}
\ds \int_{\R^{N-1}} \frac{dz'}{(|z'|^2+1)^\frac{N+2s}{2}} \hspace{-2mm} & = (N-1) \omega_{N-1} 
\ds \int_0^\frac{\pi}{2} (\sin t)^{N-2} (\cos t)^{2s} dt\\
& \ds = \frac{(N-1) \omega_{N-1}}{2} B \left( \frac{N-1}{2},s+\frac{1}{2}\right)\\[1pc]
& \ds = \frac{(N-1) \omega_{N-1}}{2} \frac{\Gamma\left(\frac{N-1}{2}\right) \Gamma \left( s+\frac{1}{2} 
\right)}{\Gamma \left( s+\frac{N}{2}\right)},
\end{array}
$$
where $B(x,y)$ is the Beta function. Next, we use a well-known expression for $\omega_{N-1}$ (cf. for instance page 9 in \cite{MR1814364}) 
to obtain that 
\begin{equation}\label{eq-expr-I}
\int_{\R^{N-1}} \frac{dz'}{(|z'|^2+1)^\frac{N+2s}{2}} = \frac{\pi^\frac{N-1}{2} \Gamma\left( s+
\frac{1}{2}\right)}{\Gamma \left( s+\frac{N}{2}\right)}.
\end{equation}
Finally, with the use of \eqref{eq-const-norm} and \eqref{eq-expr-I} we see that 
$$
c(N,s) \int_{\R^{N-1}} \frac{dz'}{(|z'|^2+1)^\frac{N+2s}{2}} =4^s s(1-s)  
\frac{\pi^{-\frac{1}{2}} \Gamma\left( s+ \frac{1}{2}\right)}{\Gamma(2-s)} = c(1,s),
$$
as was to be shown. This concludes the proof of the theorem.
\end{proof}

\bigskip

\begin{proof}[Proof of Theorem \ref{theorem:main}]
	Since we have proved the uniqueness of solutions of $(P_1)$ with the same supremum $\rho$ (see Theorem \ref{theorem:unidim}) the proof of Theorem \ref{theorem:main} will follow by showing the existence of two one-dimensional 
	solutions $\underline{u}$, $\overline{u}$ of \eqref{eq:Problema}  
	verifying $\underline{u}\leq u\leq \overline{u}$ in $\R^N$ and 
	$$
	\lim_{x_N\to +\infty} \underline{u}(x)=\lim_{x_N\to +\infty} \overline{u}(x) =\rho.
	$$
	
	\medskip
	
	\noindent {\it Step 1.} There exists a one-dimensional solution $\overline u$ of \eqref{eq:Problema} 
	with $u \le \overline u \le \rho$.
	
	\smallskip
	
	Indeed, let $\overline{u}$ be the maximal solution of \eqref{eq:Problema} relative to $\rho$ given by Theorem 
	\ref{theorem:SubSupV} in the Appendix. Then by maximality it is clear that $\overline{u}$ is one-dimensional and 
	$u\le \overline{u}\le \rho$.
	
	Observe that this implies that $f$ verifies \eqref{HF1} by Theorem \ref{theorem:unidim}.
	
	\medskip
	
	\noindent {\it Step 2.} For every $R>0$ and every $\varepsilon>0$ small enough, there 
	exists $x_0\in \R^N_+$ such that $B_R(x_0)\subset \subset \R^N_+$ and $u\ge \rho-\varepsilon$ 
	in $B_R(x_0)$.
	
	\smallskip
	
	To prove this assertion take $\{x_n\}_{n\in\mathbb{N}}\subset \R^N_+$ such that 
	$u(x_n)\to\rho$ as $n\to +\infty$. We claim that $x_{n,N}\to +\infty$ (observe that if 
	$f\in C^1(\R)$, this would follow at once from the monotonicity of $u$ in the $x_N$ 
	direction given by Theorem 1 in \cite{MR3624935}). 
	
	Arguing as in the proof of Theorem \ref{theorem:SubSupV} in the Appendix, we obtain that 
	\begin{equation}\label{cota-con-phi}
	u(x) \le A \varphi(x)  \quad \hbox{in } \{x\in \R^N_+:\ 0<x_N<1\},
	\end{equation}
	where $A>0$ and $\varphi$ is given by \eqref{def-phi}. Since $\varphi=0$ on $\partial \R^N_+$, 
	this actually shows that $x_{n,N}$ is bounded away from zero, so extracting a subsequence 
	we may assume that either $x_{n,N}\to \mu$ for some $\mu>0$ or $x_{n,N}\to +\infty$. Define
	$$
	u_n(x)=u(x+x_n) \quad x\in \R^N.
	$$
	Proceeding as in previous situations, we obtain that, passing to a subsequence $u_n\to v$ 
	locally uniformly in $\R^N$, where $v$ is a solution of 
	$$
	(-\Delta)^s v= f(v) \qquad \hbox{in } D.
	$$
	Here $D=\{x\in \R^N:\ x_{n,N}>-\mu\}$ in case $x_{n,N}\to \mu$ or $D=\R^N$ when $x_{n,N}\to +\infty$. 
	In either case, and using that $f(\rho)=0$ and the Lipschitz condition on $f$, the strong 
	maximum principle implies $v\equiv \rho$. However, from \eqref{cota-con-phi} we have in the former 
	case
	$$
	v(x) \le A \varphi(x+\mu e_N) \quad \hbox{in } \{x\in \R^N_+:\ -\mu<x_N<0\}
	$$
	which would yield that $v$ vanishes on $\partial  D$, impossible. Hence the latter possibility holds 
	and this shows $x_{n,N}\to +\infty$ and $u(x+x_n) \to \rho$ locally uniformly in $\R^N$.
	
	Finally, let $\varepsilon>0$ and $R>0$ be arbitrary. We have $u(x+x_n)\ge \rho-\varepsilon$ in $B_R$ if $n$ is larger 
	than some $n_0=n_0(\varepsilon,R)$. Then $u(x) \ge \rho-\varepsilon$ in $B_R(x_n)$ for those values of $n$, as was 
	to be shown.

	\medskip
	
	\noindent {\it Step 3.} For every $\eta>0$, there exists $c(\eta)>0$ such that 
		\begin{equation}\label{sliding0}
		u(x)\geq c (\eta)
		\text{ when } x_N\geq \eta.
	\end{equation}
	
	\smallskip
	
	Indeed, let $\varepsilon>0$. When $f(0)<0$, choose a small positive $\delta$ such that 
	\eqref{fdelta} is verified, otherwise set $\delta=0$. Recall that by Step 1 $f$ verifies 
	\eqref{HF1}. Thus we may apply Lemma \ref{lemma:EBDM}: there exists $R_0=R_0(\varepsilon,\delta)$ such that 
	the maximal solution $u_{R_0,\delta}$ of \eqref{eq:PBD2} verifies 	
	\begin{equation}\label{Dud1}
		\|u_{R_0,\delta}\|_{\scriptstyle L^\infty(B_{R_0})}
		=\rho-\varepsilon.
	\end{equation}
	Let $x_0\in \R^N_+$ be given in Step 2 above, for these particular values of $\varepsilon$ and 
	$R_0$. Then by \eqref{Dud1}
		$$
		u_{R_0,\delta}(z-x_0) \le u(z), \qquad z\in B_{R_0}(x_0).
	  $$
	Since $u\ge 0$ in $\R^N$ and $u_{R_0,\delta}(\cdot-x_0)=-\delta \le 0$ outside 
	$B_{R_0}(x_0)$, we also have	
	\begin{equation}\label{Dud4}
		u_{R_0,\delta}(z-x_0)\leq u(z),\qquad  z\in\mathbb{R}^{N}.
	\end{equation}
	On the other hand, recall that by Lemma \ref{lemma:EBDM}, $u_{R_0,\delta}$ 
	is radially symmetric and radially decreasing. Hence there exists 
	$R_1\in(0,R_0]$ such that the set of points where $u_{R_0,\delta}>0$ is 
	precisely $B_{R_1}$. 
	Denote $\Theta_{R_1}\coloneqq \{x\in\mathbb{R}^{N}\colon x_{N} > R_1\}$, and 
	consider the set 
	\[
		\Omega_{R_1}\coloneqq\{x\in \Theta_{R_1}:\, u_{R_0,\delta}(z-x)<u(z),\, 
		z\in\mathbb{R}^N_+\}.
	\]
	It follows by \eqref{Dud4} and the strong maximum principle that $x_0\in \Omega_{R_1}$, hence this 
	set is nonempty. 
	We now claim that $\Omega_{R_1}$ is both open and closed relative to $\Theta_{R_1}$, 
	therefore 
	\begin{equation}\label{Trmp}
		\Omega_{R_1}=\Theta_{R_1}.
	\end{equation}
	Indeed it is clear from the continuity of all functions involved that $\Omega_{R_1}$ is open. 
	As for the closedness, if $\{\xi_n\} \subset \Omega_{R_1}$ verifies $\xi_n\to \xi\in \Theta_{R_1}$, then 
	$u_{R_0,\delta}(z-\xi_n)\le u(z)$ in $\R^N$, and by the strong maximum principle and the positivity of 
	$u$, this inequality is strict in $\R^N_+$, hence $\xi\in \Omega_{R_1}$.
	We deduce that \eqref{Dud4} holds for every $x$ with $x_N\ge R_1$.
	
	\smallskip
	
	Finally, let $\eta>0$ and take $0<\varepsilon < \min\{\eta,R_1\}$ fixed but arbitrary. 
	If $z\in \R^N_+$ is such that $z_N\ge \eta$,  
	it easily follows that $z\in B_{R_1-\varepsilon}(x_z)$, 
	where $x_z:=(z',R_1+z_N-\varepsilon)\in \Theta_{R_1}$.
	Therefore, by \eqref{Dud4} we see that 
	$$
	u(z)\ge  c(\eta)\coloneqq\inf
		\left\{ u_{R_0,\delta}(x)\colon x\in 
		B_{R_1-\varepsilon} \right\}>0,
	$$
	which concludes the proof of Step 3.
	
	\medskip
	
	\noindent {\it Step 4.} For every $M>2R_0$ and $\nu<M-2R_0$, there exists a maximal 
	solution $u_{\nu,M}$ of the problem 
		\begin{equation}\label{eq:prob-strip}
		\begin{cases}
			(-\Delta)^s u = f(u) 
			&\text{ in } \Sigma_{\nu,M}
			\coloneqq\{x\in\mathbb{R}^{N}\colon 
			\nu<x_{N}< M\} ,\\
			\ \	u=0 &\text{ in } \mathbb{R}^{N}\setminus 
				\Sigma_{\nu,M},
			\end{cases}
		\end{equation}
	relative to $u$, which only depends on $x_N$ and verifies $\| u_{\nu,M} \|_{L^\infty(\R^N)}\ge \rho-\varepsilon$.
	
	\smallskip
	
	Consider the maximal solution $\widetilde{u}_{R_0,\delta}$ of problem \eqref{eq:PBD2}. If we choose, say, 
	$x_0=(0,\frac{M}{2})$, then the function $\widetilde{u}_{R_0,\delta}(x-x_0)$ is a subsolution of 
	\begin{equation}\begin{cases}\label{eq:prob-strip-delta}
			(-\Delta)^s u = f_{\delta}(u) 
			&\text{ in } \Sigma_{\nu,M}
			\coloneqq\{x\in\mathbb{R}^{N}\colon 
			\nu<x_{N}< M\} ,\\
			\ \	u=0 &\text{ in } \mathbb{R}^{N}\setminus 
				\Sigma_{\nu,M},
			\end{cases}
			\end{equation}
	while $u$ is a supersolution, and they are ordered because of Step 3. The existence of a maximal solution $u_{\nu,M,\delta}$ 
	of \eqref{eq:prob-strip-delta} relative to $u$ then follows directly by Theorem \ref{theorem:SubSupV} 
	in the Appendix (cf. also Remark \ref{rem-franja}). 
	It is clear that $\| u_{\nu,M,\delta} \|_{L^\infty(\R^N)}\ge \rho-\varepsilon$. Proceeding as in the proof of 
	Theorem \ref{theorem:MaximalSolR}, passing to the limit when $\delta\to 0^+$, we get the existence of a maximal 
	solution $u_{\nu,M}$ of \eqref{eq:prob-strip}. 
	
	Thus only the one-dimensional symmetry of $u_{\nu,M}$ remains to be shown. For this aim 
	we will first show that for every unitary vector $\theta\in\mathbb{R}^{N-1}$ and $\lambda>0$ 
	\begin{equation}\label{sliding2}
		u_{\nu,M}(x'+\lambda \theta,x_N)\leq u(x)
		\quad x\in\mathbb{R}^{N}.
	\end{equation}
 	The proof of this statement is a consequence again of the sliding method. However, we 
	should warn that it is not completely standard since now we are sliding with solutions which do 
	not have a compact support as in most previous situations (see for instance \cite{MR1159383}).
 	
	Fix a unitary vector $\theta \in \R^{N-1}$. We will see that \eqref{sliding2} holds for small $\lambda$. 
	If it were not true, then there would 
 	exist sequences $\lambda_n\to 0^+$ and $\{x_n\}_{n\in\mathbb{N}}
 	\subseteq\Sigma_{\nu,M}$ such that
	\begin{equation}\label{sliding3}
		u_{\nu,M}(x_n'+\lambda_n \theta,x_{n,N})
		\geq u(x_n) \quad  n\in\mathbb{N}.
	\end{equation}
	We may assume with no loss of generality that $x_{0,N}\to x_0\in [\nu,M]$. 
	If we now define the translated functions
	\[
		u_{\nu,M,n}(x)\coloneqq u_{\nu,M}(x'+x_n', x_N), 
		\quad u_{n}(x)\coloneqq u(x'+x_n', x_N), 
		\quad x\in\mathbb{R}^{N},
	\]
	we can proceed similarly as in previous situations to obtain that, 
	up to extraction of a subsequence, $u_{\nu,M,n}\to U_{\nu,M}$ and $u_{n}  \to U$
	uniformly on compact sets of $\R^N$ as $n\to +\infty$ where
	$U_{\nu,M}$ and $U$ are solutions of \eqref{eq:prob-strip} and \eqref{eq:Problema}, respectively. 
	On the other hand since, by construction, $u_{\nu,M}(x)\leq u(x)$, 
	for any $x\in\mathbb{R}^N,$ we have   
	\[
		U_{\nu,M}(x)\leq U(x) \quad x\in \R^N.
	\]
	Then, by \eqref{sliding3} we deduce
	\begin{equation}\label{sliding4}
		U_{\nu,M}(0,x_0)= U(0,x_0).
	\end{equation} 
	Observe that, by \eqref{sliding0}, we have 
	\begin{equation}\label{contrad}
	\mbox{$U\ge c(\nu)>0$ on $\partial \Sigma_{\nu,M}$ 
	while $U_{\nu,M}=0$ there.}
	\end{equation}
	Therefore $(0,x_0) \in \Sigma_{\nu,M}$ and by 
	\eqref{sliding4} and the strong maximum principle, we can conclude that $U_{\nu,M}=U$ 
	in $\R^N$. However, this is impossible by \eqref{contrad}. 
	Therefore \eqref{sliding2} is true for small enough $\lambda>0$.
	
	Next, define
	\[
		\lambda^{*}\coloneqq
		\sup\{\mu>0\colon \eqref{sliding2} 
		\text{ holds for every } \lambda\in (0,\mu)\},
	\]
	and assume $\lambda^*<+\infty$. By continuity we have 
	$u_{\nu,M}(x'+\lambda^{*} \theta,x_N)\leq u(x)$ 
	for any $x\in\mathbb{R}^N$, and we reach a contradiction arguing exactly as before. 
	The contradiction shows that $\lambda^*=+\infty$, that is, \eqref{sliding2} holds 
	for every $\lambda>0$ and every unitary $\theta \in \R^{N-1}$.	
	
	Finally, since $u_{\nu,M}(x'+\lambda^{*} \theta,x_N)$ is a solution of problem 
	\eqref{eq:prob-strip} which lies below $u$, we see by maximality that 
	$$
	u_{\nu,M}(x'+\lambda^{*} \theta,x_N) \le u_{\nu,M}(x) \qquad x\in \R^N.
	$$
	Since $\lambda>0$ and $\theta\in \R^{N-1}$ are arbitrary, this shows that $u_{\nu,M}$ 
	depends only on $x_N$.
	
	\medskip
	
	\noindent {\it Step 5}. There exists a one dimensional solution $\underline{u}$ of 
	\eqref{eq:Problema} verifying $\| \underline{u}\|_{L^\infty(\R^N)}=\rho$ and 
	$\underline{u}\le u$ in $\R^N$.
	
	\smallskip
	
	By a similar argument as in Remark \ref{remark:Order}, we see that $u_{\nu,M}$ is decreasing in $\nu$ 
	and increasing in $M$. Proceeding as in the proof of Theorem \ref{theorem:MaximalSolR}, we see that 
		$$
		\underline{u}_\varepsilon(x)
		\coloneqq \lim_{\nu\to 0}\lim_{M\to +\infty}
		 u_{\nu,M}(x_N),\quad x\in\mathbb{R}^N
	  $$ 
	is a nonnegative one-dimensional solution of $\eqref{eq:Problema}$, which verifies  
	$\underline{u}_\varepsilon\le u$ in $\mathbb{R}^N$ and 
	$\| \underline{u}_\varepsilon\|_{L^\infty(\R^N)}\ge \rho-\varepsilon$. Moreover, it can be 
	checked that $\underline{u}_\varepsilon$ is increasing in $\varepsilon$ as $\varepsilon \to 0^+$. 
	Therefore
		$$
		\underline{u} \coloneqq \lim_{\varepsilon\to 0^+} \underline{u}_\varepsilon(x),
		\quad x\in\mathbb{R}^N
	  $$
		is a nonnegative one-dimensional solution of $\eqref{eq:Problema}$, which verifies  
	$\underline{u}\le u$ in $\mathbb{R}^N$ and 
	$\| \underline{u} \|_{L^\infty(\R^N)}= \rho$. 
	
	\medskip
	
	\noindent {\it Completion of the proof}. By Theorem \ref{theorem:unidim} we have that $\underline{u}= \overline{u}=
	u_{\rho}$.  Then by Theorem \ref{theorem:unidim} and Proposition \ref{theorem:exist-half} $u$ coincides with $u_{\rho}$. 
\end{proof}

\bigskip

\begin{proof}[Proof of Theorem \ref{theorem:valdinoci}]
The first step is to show that $\|u\|_{L^\infty(\R^N)} \le \rho$. Assume on the 
contrary that the set $D:=\{x\in \R^N_+:\ u(x)>\rho\}$ is nonempty. Then the function 
$v=\rho-u$ verifies
$$
\begin{cases}
(-\Delta)^s v=- f(u)\ge 0 &\text{ in } D,\\
\ \ v\ge 0 &\text{ in } \mathbb{R}^N\setminus D.
\end{cases}
$$
We can use Lemma 4 in \cite{MR3624935} to deduce that $v\ge 0$ in $D$, that is  
$u\le \rho$ in $D$, which is a contradiction ({notice that the requirement in \cite{MR3624935} 
that $D$ is connected}, which we can not ensure in our situation, is not really necessary). 
The contradiction shows that $u\le \rho$.

\medskip

The rest of the proof is entirely similar to that of Theorem \ref{theorem:main}. 
Indeed, the existence of a one-dimensional solution $\overline{u}$ of \eqref{eq:Problema} 
verifying $u\le \overline{u}$ in $\R^N$ follows exactly the same way. 

As for the existence of a one-dimensional solution $\underline{u}$ of \eqref{eq:Problema} 
verifying $\underline{u} \le u$ in $\R^N$, we notice that Step 2 is no longer needed and 
Step 3 can be directly proved with the use of the sliding method, as in 
\cite{MR1470317,2016arXiv160407755S}. Indeed we  claim that for every $\eta>0$ there exists $c(\eta)>0$ such that 
\begin{equation}\label{eq-sliding-2}
u(x) \ge c(\eta) \quad \hbox{if } x_N\ge \eta.
\end{equation}
To see this, we use hypothesis \eqref{hip-cero}: there exist $c,\nu>0$ such that 
$f(t)\ge ct$ if $t\in [0,\nu]$. Choose $R>0$ so that the first eigenvalue of $(-\Delta)^s$ in 
$B_R$ verifies $\lambda_1(B_R)\le c$, and let $\phi$ be an associated positive eigenfunction 
normalized by $\|\phi\|_{L^\infty(B_R)}=1$. 
Then it is clear that for every $x_0$ such that $x_{0,N}> R$ the function 
$$
\underline{u} (x)=\delta \phi(x-x_0), \quad x\in \R^N,
$$
is a subsolution of \eqref{eq:Problema} when $0< \delta \le \nu$. Moreover, if we fix such an 
$x_0$ it is possible to choose a small enough $\delta$ to have in addition $\underline{u}\le u$ 
in $\R^N$. Indeed, this inequality is trivially satisfied outside $B_R(x_0)$, while in 
$B_R(x_0)$ the inequality is also true for small $\delta$ because $u$ is bounded 
away from zero there.

We can now `slide' the ball around $\R^N_+$ just like in Step 3 in the proof of Theorem 
\ref{theorem:main} to obtain \eqref{eq-sliding-2}. Arguing as in Step 4 there, we can now construct 
a one-dimensional solution $\underline{u}$ of \eqref{eq:Problema} verifying $\underline{u}\le u$ 
in $\R^N$. Finally, observe that by Theorem \ref{theorem:unidim}, problem \eqref{eq:Problema} admits a unique 
one-dimensional solution given by $u_{\rho}$. Therefore, $u=u_{\rho}$, as we wanted to show. 
\end{proof}

\bigskip

The proof of our last result is just a direct consequence of Theorem \ref{theorem:valdinoci}.

\medskip

\begin{proof}[Proof of Corollary \ref{Liouville}]
Assume there exists a bounded, positive solution of \eqref{eq:Problema} and let $\rho_0:=\|u\|_{L^\infty(\R^N)}$. 
We choose $\rho>\rho_0$ and modify $f$ in the interval $(\rho_0,\rho)$ in such a way that 
$f$ remains positive in $(0,\rho)$ and $f(\rho)=0$. It is clear that \eqref{HF1} is verified 
for the value of $\rho$ so chosen. Hence, by Theorem \ref{theorem:valdinoci}, problem 
\eqref{eq:Problema} admits a unique solution $v$ which is one-dimensional and verifies 
$$
\lim_{x\to +\infty} v(x)=\rho.
$$
By uniqueness we should have $u\equiv v$ in $\R^N$, but this is impossible as $\| u\|_{L^\infty(\R^N)}=
\rho_0<\rho$. This contradiction shows that problem \eqref{eq:Problema} does not admit any 
bounded, positive solution, as we wanted to show.
\end{proof}

\bigskip


\appendix
\setcounter{equation}{0}

\renewcommand{\theequation}{A.\arabic{equation}}

\section{A solution between a sub and a supersolution}

In this Appendix we collect a couple of results which deal with the existence of 
maximal solutions for some problems related to the ones considered in the paper. 
To begin with, let $\Omega$ be a bounded domain and $f\colon\overline{\Omega}\times\mathbb{R}\to\mathbb{R}$ 
be a continuous function. We consider 
	\begin{equation}\label{eq:ENV}
		\begin{cases}
	 		(-\Delta)^s u=f(x,u) & \text{ in }\Omega,\\
			\ \ u=g					 & \text{ in }\mathbb{R}^N
			\setminus\Omega,
		\end{cases}
	\end{equation}
	where $g\in C(\R^N)$. 
	
	For convenience, we only deal in this Appendix with subsolutions, supersolutions and solutions 
	in the viscosity sense. In some cases, however, it is known that with some requirements on $f$ and $g$ 
	the concepts of viscosity and classical solutions of \eqref{eq:ENV} coincide (see 
	\cite{MR2270163,MR2494809,MR2781586}). 
	
	\medskip
	
	We say that a function $u\in C(\R^N)$ is a viscosity subsolution of \eqref{eq:ENV} if 
	$u\le g$ in $\R^N\setminus \Omega$ and verifies the following: for any $x_0\in\Omega$ and any
	function $\phi$ which is $C^2$ in a neighbourhood $U$ of $x_0$ and such that 
	$u(x_0)=\phi(x_0)$ and $u\le\phi$ in $U$ we have $(-\Delta)^s v(x_0)\le f(x_0,v(x_0))$, 
	where
	\[
		v(x)\coloneqq		
		\begin{cases}
			\phi(x) &\text{ if } x\in U,\\
			u(x)	 &\text{ if } x\in \mathbb{R}^N\setminus U.\\
		\end{cases}
	\]
	Supersolutions are defined by reversing the above inequalities. A function $u$ is 
	a viscosity solution of \eqref{eq:ENV} it it is both a viscosity sub and supersolution of \eqref{eq:ENV}.
	We remark that the continuity assumption on both the sub and supersolution can be relaxed to 
	an appropriate lower semicontinuity, but we are only interested in this work in continuous sub and 
	supersolutions.

	The existence of a solution between a sub and a supersolution is well-known in several instances, 
	mainly when an iteration procedure is available. However, that a maximal solution can be obtained in 
	general is perhaps less known, so we will include a sketch of the main proofs. Given a viscosity 
	supersolution $\overline{u}$ we say that $u$ is the maximal solution relative to $\overline{u}$ if 
	for every other viscosity solution $v$ of \eqref{eq:ENV} verifying $v\le \overline{u}$ in $\R^N$ 
	we have $v\le u$ in $\R^N$.
		
	Then we have: 
	
	\begin{theorem}\label{theorem:SubSup}
		Let $\Omega$ be a bounded Lipschitz domain satisfying the exterior sphere condition. Assume 
		$f:\overline{\Omega}\times \R\to \R$ is continuous and that $g\in C(\R^N)\cap L^\infty(\R^N)$. 
		If there exist viscosity sub and supersolution $\underline{u},\overline{u}\in C(\mathbb{R}^N)\cap L^\infty(\R^N)$ 
		of \eqref{eq:ENV} with $\underline{u}\le\overline{u}$ in $\mathbb{R}^N$, then there exists a 
		maximal viscosity solution $\widetilde{u}$ of \eqref{eq:ENV} relative to $\overline{u}$. 
	\end{theorem}
	
	\begin{proof}[Sketch of proof]		
	Let us begin by observing that the problem can be reduced to $g=0$: if $w$ is the unique $s-$harmonic 
	function in $\Omega$ which coincides with $g$ outside $\Omega$ and we let $v=u-w$, then $v$ is a solution 
	of \eqref{eq:ENV} with right hand side $h(x,v)= f(x,v+w(x))$ and vanishing outside $\Omega$. Thus we 
	may assume in what follows that $g=0$.

	The existence of a solution in the interval $[\underline{u},\overline{u}]$ follows exactly as in 
	Theorem 1 in \cite{MR988383} (we notice that only regularity theory and the maximum principle are 
	needed; see also \cite{MR3393247}). Thus we only show the existence of a maximal solution in this 
	interval. We define the non-empty set
			\[
			\mathcal{F}\coloneqq\!\left\{
			u\in C(\mathbb{R}^N)\colon
			u \text{ is a viscosity solution of \eqref{eq:ENV} such that } \underline{u}\le u \le \overline{u} \text{ in } \R^N \right\}
		\]
		and 
		\[
			\widetilde{u}(x) \coloneqq\sup\{u(x)\colon u\in \mathcal{F}\}.
		\]
		We observe that for every $u\in \mathcal{F}$ we have $\|u\|_{L^\infty(\R^N)}\le C$, 
		$\|f(\cdot, u)\|_{L^\infty(\Omega)}\le C$, for some positive constant $C$ which does 
		not depend on $u$. Thus by regularity theory (cf. for instance Proposition 1.1 in 
		\cite{MR3168912}), we obtain 
		$$
		\| u \|_{C^s(\overline{\Omega})} \le C.
		$$
		This means that the set $\mathcal{F}$ is equicontinuous, thus $\widetilde{u}$ is 
		continuous in $\R^N$ and vanishes outside $\Omega$. Moreover, it is well-known that $\widetilde{u}$ 
		is a subsolution of \eqref{eq:ENV} in the viscosity sense. 
		
		Thus there exists a solution of \eqref{eq:ENV} in the interval $[\widetilde{u}, 
		\overline{u}]$. By its very definition it follows that this solution is indeed 
		$\widetilde{u}$, which is clearly the maximal solution in the interval 
		$[\underline{u},\overline{u}]$. 
		Let us mention in passing that the existence of the maximal solution could also be shown 
		by following the approach in \cite{MR2390510}.	
		
		We finally show that the maximal solution $\widetilde{u}$ just obtained 
does not depend on $\underline{u}$. Indeed, assume $\underline{u}_1$ and $\underline{u}_2$ are subsolutions of 
\eqref{eq:ENV} which lie below the supersolution $\overline{u}$ in $\R^N$. Let $\widetilde{u}_i$ 
be the maximal solution in the interval $[\underline{u}_i,\overline{u}]$, $i=1,2$ and set 
$\underline{u}^+= \max\{\widetilde{u}_1,\widetilde{u}_2\}$. 
Then $\underline{u}^+$ is a subsolution of \eqref{eq:ENV} below $\overline{u}$. Thus there exists 
a solution $w$ verifying $\underline{u}^+\le w \le \overline{u}$ in $\R^N$. In particular 
$\underline{u}_i\le \widetilde{u}_i \le w \le \overline{u}$ in $\R^N$, $i=1,2$ and by maximality 
of $\widetilde{u}_i$ we deduce $\widetilde{u}_1=\widetilde{u}_2=w$ in $\R^N$.
	\end{proof}
	
	\bigskip

	Theorem \ref{theorem:SubSup} can be generalized to deal with unbounded domains $\Omega$. Only some 
		minor points in the proof above need to be especially treated. For simplicity we will restrict 
		our attention next to the case $\Omega=\R^N_+$ and $f$ not depending on $x$, which is the main concern 
		in this paper:
		\begin{equation}\label{eq:ENV-semiesp}
		\begin{cases}
	 		(-\Delta)^s u=f(u) & \text{ in }\R^N_+,\\
			\ \ u=0					 & \text{ in }\mathbb{R}^N\setminus\R^N_+.
		\end{cases}
	\end{equation}
	We have also set $g=0$. 
	In this context, we have a result which is completely analogue to Theorem \ref{theorem:SubSup}.
		
	\begin{theorem}\label{theorem:SubSupV}
		Assume $f:\R \to \R$ is continuous and there exist viscosity 
		sub and supersolution $\underline{u},\overline{u}\in C(\mathbb{R}^N)\cap L^\infty(\R^N)$ 
		of \eqref{eq:ENV-semiesp} with $\underline{u}\le \overline{u}$ in $\R^N$. 
		Then there exists a maximal viscosity solution of \eqref{eq:ENV-semiesp} relative to $\overline{u}$.
	\end{theorem}

\begin{proof}[Sketch of proof]
First of all we truncate $f$ outside $[\inf\underline{u},\sup \overline{u}]$ to make it bounded.
Let $\varphi$ be the solution of the one-dimensional problem
\begin{equation}\label{def-phi}
	\begin{cases}
	 		(-\Delta)^s \varphi=1 & \text{ in }(0,1) ,\\
			\ \ \varphi=0					 & \text{ in } (-\infty,0),\\
			\ \ \varphi=1					 & \text{ in } (1,+\infty).\\
		\end{cases}
	\end{equation}
	Then $\varphi$ solves the same problem in $\Sigma_1=\{x\in \R^N_+:\ 0<x_N<1\}$ (cf. the proof of 
	Proposition \ref{theorem:exist-half} in Section \ref{sec-main}). We notice that for large enough $c>0$ 
	the function $-c\varphi$ (resp. $c\varphi$) is a subsolution (resp. supersolution) of \eqref{eq:ENV-semiesp}. Therefore $\underline{v}:=\max\, \{\underline{u}, -c\varphi\}$ and $\overline{v}:=\min\, \{\overline{u}, c \varphi\}$ are well ordered sub and supersolution of \ref{eq:ENV-semiesp} satisfying $\underline{v}=\overline{v}=0$ in $\mathbb{R}^N\setminus\R^N_+$.

We choose now any smooth function $w$ defined in $\R^N$ and verifying $w=0$ in $\R^N\setminus \R^N_+$, 
$\underline{v}\le w \le \overline{v}$ in $\R^N$. For $R>0$, let $B_R^+ =\{x\in \R^N_+:\ |x|<R\}$ 
and consider the problem 
	\begin{equation}\label{eq:truncada}
		\begin{cases}
	 		(-\Delta)^s u=f(x,u) & \text{ in }B_R^+ ,\\
			\ \ u=w					 & \text{ in }\mathbb{R}^N
			\setminus B_R^+.
		\end{cases}
	\end{equation}
	By Theorem \ref{theorem:SubSup}, there exists a solution $u_R$ of \eqref{eq:truncada} verifying 
	$\underline{v}\le u_R \le \overline{v}$ in $\R^N$. Moreover, 
	the family $\{u_R\}_{R>0}$ is uniformly bounded and by standard interior regularity we also have
	$$
	\| u_R\|_{C^s(K)}\le C,
	$$ 
	for every compact set $K\subset \R^N_+$. Thus $\{u_R\}_{R>0}$ is also equicontinuous and we can 
	select a sequence $R_n\to +\infty$ such that $u_{R_n}\to v$ locally uniformly in $\R^N_+$ for some 
	function $v\in C(\R^N)$ which verifies $\underline{v}\le v \le \overline{v}$ in $\R^N$, therefore 
	vanishes in $\R^N\setminus \R^N _+$. 
	
	Passing to the limit in \eqref{eq:truncada} we obtain that $v$ is a solution of \eqref{eq:ENV-semiesp} 
	which verifies $\underline{u}\le v\le \overline{u}$ in $\R^N$. The existence of a maximal 
	solution relative to $\overline{u}$ is shown exactly as in the proof of Theorem \ref{theorem:SubSup}, with 
	the only prevention that the barrier $\varphi$ constructed above has to be used instead of 
	the boundary regularity for bounded domains.
\end{proof}

\bigskip	

\begin{remark}\label{rem-franja}{\rm \ 
(a) Of course the same result is true when $N=1$, in particular for problem 
\eqref{eq:Problema-unidim}.

\smallskip

(b) With a minor variation in the proof of Theorem \ref{theorem:SubSupV} it can be seen that the 
same statements hold when problem \eqref{eq:ENV-semiesp} is posed in a strip 
$\Sigma_{\nu,M}\coloneqq\{x\in\mathbb{R}^{N}\colon \nu<x_{N}< M\}$, where $M>\nu>0$.
}\end{remark}

\noindent {\bf Acknowledgements.} 
	All authors were partially supported by Ministerio de Eco\-no\-m\'ia y 
	Competitividad under grant MTM2014-52822-P (Spain). 
	B. B. was partially supported by a MEC-Juan de la Cierva postdoctoral 
	fellowship number  FJCI-2014-20504 (Spain). 
	L. D. P. was partially supported by PICT2012 0153 from ANPCyT 
	(Argentina).
	A. Q. was partially supported by Fondecyt Grant No. 1151180 Programa Basal, 
	CMM. U. de Chile and Millennium Nucleus Center for Analysis of 
	PDE NC130017. 


\bibliographystyle{abbrv}

\bibliography{Bibliografia}

\end{document}